\documentclass[a4paper,12pt]{article}
\usepackage{tikz}
\usepackage{float}
\usepackage{pgfplots}
\usetikzlibrary{calc}
\usepackage{amssymb}
\usetikzlibrary{quotes,angles}
\usetikzlibrary{plotmarks}
\usetikzlibrary{arrows,shapes,positioning}
\usetikzlibrary{decorations.markings}
\usepackage{geometry}
\usepackage{caption}
\tikzstyle arrowstyle=[scale=2]
\tikzstyle directed=[postaction={decorate,decoration={markings,
    mark=at position .65 with {\arrow[arrowstyle]{stealth}}}}]
\tikzstyle reverse directed=[postaction={decorate,decoration={markings,
    mark=at position .65 with {\arrowreversed[arrowstyle]{stealth};}}}]
    \tikzstyle left directed=[postaction={decorate,decoration={markings,
    mark=at position -.62 with {\arrow[arrowstyle]{stealth}}}}]

\tikzstyle left reverse directed=[postaction={decorate,decoration={markings,
    mark=at position -.62 with {\arrowreversed[arrowstyle]{stealth};}}}]
\usepackage{indentfirst}
\usepackage[plainpages=false]{hyperref}
\usepackage{amsfonts,latexsym,rawfonts,amsmath,amssymb,amsthm,mathrsfs}
\usepackage{amsmath,amssymb,amsfonts,latexsym,lscape,rawfonts}

\usepackage[all]{xy}
\usepackage{eufrak}
\usepackage{makeidx}         
\usepackage{graphicx,psfrag}
\usepackage{array,tabularx}

\usepackage{setspace}
\newtheorem{thm}{Theorem}[section]
\newtheorem{cor}[thm]{Corollary}
\newtheorem{lem}[thm]{Lemma}
\newtheorem{clm}[thm]{Claim}
\newtheorem{prop}[thm]{Proposition}

\theoremstyle{remark}
\newtheorem{rmk}[thm]{Remark}

\theoremstyle{definition}
\newtheorem{Def}[thm]{Definition}                                        %

\usepackage{tikz}
\usetikzlibrary{calc}
\usepackage{amssymb}
\usetikzlibrary{quotes,angles}
\usepackage{tocvsec2}
\usetikzlibrary{plotmarks}
\usetikzlibrary{arrows,shapes,positioning}
\usetikzlibrary{decorations.markings}
\usepackage{geometry}
\usepackage{appendix}

\title{Ricci-flat  manifolds  of generalized ALG asymptotics}

{\small{\author{Yuanqi Wang\thanks{University of Kansas, Lawrence, KS, USA.\ yqwang@ku.edu.}}}

\date{\vspace{-5ex}}

\begin{document}

\maketitle

\begin{abstract} In  complex dimensions $\geq 3$,  we provide a geometric existence for generalized ALG complete non-compact Ricci flat K\"ahler manifolds  with Schwartz decay i.e. metric decay in any polynomial rate to an ALG model $\mathbb{C}\times Y$ modulo finite cyclic group action, where $Y$ is Calabi-Yau. 

Consequently,  for any $K3$ surface with a purely non-symplectic automorphism $\sigma$ of finite order, a K\"ahler crepant resolution of the orbifold 
$\frac{\mathbb{C} \times K3}{\langle  \sigma \rangle}$
admits  ALG Ricci-flat K\"ahler metrics with Schwartz  decay. It is known that K\"ahler crepant resolution exists in our case. Hence there are  $39$ integers,  such that $2\pi$ divided by each of them is the asymptotic angle of an ALG Ricci-flat K\"ahler $3-$fold with Schwartz decay. We also exhibit a 1638 parameters family of  ALG Ricci-flat  K\"ahler $3-$folds with asymptotic angle $\pi$ that realize $64$ distinct triples of Betti numbers. They are  iso-trivially fibred by $K3$ surface with a non-symplectic Nikulin involution.

A simple version of local K\"unneth formula for $H^{1,1}$/local $i\partial\overline{\partial}-$lemma plays a role in both the Schwartz decay, and the construction of  ansatz that equals a Ricci flat ALG model outside a compact set (isotrivial ansatz). 

The proof of Schwartz decay relies on a non-concentration of the Newtonian potential, and can not be immediately generalized to fibration with higher dimensional base, due to existence of concentrating sequence of $L^{2}$ normalized eigen-functions  on unit round spheres of (real) dimension $\geq 2$.
\end{abstract}
\tableofcontents
\addtocontents{toc}{\protect\setcounter{tocdepth}{1}}
\section{Introduction}
\subsection{Overview}

ALG gravitational instantons 
exist in abundance and are intensively studied. For example,  see the construction by  Cherkis-Kapustin \cite{CK} via moduli of periodic monopoles, Biquard-Minerbe \cite{BM} by  Kummer gluing constructions, and Hein \cite{HeinJAMES} via  Tian-Yau type \cite{TY1} Monge-Amper\`e method  and  rational elliptic surfaces with a singular fiber removed.  Along another thread,  gluing construction for manifolds with holonomy $G_{2}$ requires quasi
asymptotic locally  Euclidean  and asymptotic cylindrical Ricci flat K\"ahler manifolds as building blocks \cite{Joyce, JK, Kovalev, KL, CHNP, CHNPDuke}.  This is analogous to that gravitational instantons are singularity models of Gromov-Hausdorff limits of  hyper-K\"ahler metrics on $K3$ surfaces (for example,  see \cite{LS,  CC,  Foscolo,  HSVZJAMES} and the references therein).

Motivated by the above,  we provide examples of generalized ALG  Ricci flat K\"ahler $3-$folds on iso-trivial $K3-$fibrations. We apply Tian-Yau Monge-Ampere solvability \cite{TY1} and Hein's version \cite{HeinThesis, HeinJAMES}. These metrics are asymptotic to a group quotient of $\mathbb{C}\times K3$, where the $K3$ fiber has  a purely non-symplectic automorphism of finite order. 

 \begin{Def}\label{Def Y sigma before thm} Let $(Y,\sigma)$ be a compact Calabi-Yau manifold with a purely non-symplectic automorphism of finite order such that  
$\sigma^{\star}\Omega_{Y}=\zeta_{n}\Omega_{Y}$, where  $\zeta_{n}$ is a  primitive $n-$th unit root, and $\Omega_{Y}$ is the now-where vanishing holomorphic $(d-1,0)-$form on $Y$ (complex dimension of $Y$ is assumed to be $d-1$). Let $\sigma$ act on $\mathbb{C}$ by multiplication of $\frac{1}{\zeta_{n}}$. Here, a Calabi-Yau manifold means a smooth irreducible projective variety with trivial canonical bundle. 
\end{Def}

Our main geometric existence of ALG Ricci flat metric on crepant resolutions works in general dimensions. 
\begin{thm}\label{Thm} Under Definition \ref{Def Y sigma before thm} and \ref{Def WHF} below,   suppose  $\widetilde{\frac{\mathbb{C} \times Y}{\langle  \sigma \rangle}}$  is  a K\"ahler crepant resolution of the orbifold $\frac{\mathbb{C} \times Y}{\langle  \sigma \rangle}$  and  $b^{1}(Y)=0$. For  any  K\"ahler metric $\omega$ on $\widetilde{\frac{\mathbb{C} \times Y}{\langle  \sigma \rangle}}$,  there is a smooth function $\phi_{RFALG}$ such that   
\begin{equation}\label{equ RFALG potential}\omega_{RFALG}\triangleq \omega+i\partial\overline{\partial}\phi_{RFALG}\end{equation}
is an   $ALG_{\infty}(\frac{2\pi}{n})$ Ricci flat metric  (with Schwartz decay). 

Suppose $\omega^{\prime}=\omega_{RFALG}+i\partial\overline{\partial}\phi$ for some bounded $C^{2}$ real valued function $\phi$. Suppose $\omega^{\prime}$ is K\"ahler,  quasi-isometric to $\omega_{RFALG}$, and has the same volume form as  $\omega_{RFALG}$ (therefore $\omega^{\prime}$ is still Ricci flat). Then it must equal $\omega_{RFALG}$, and $\phi$ must be a constant. 
 \end{thm} 
 
\begin{rmk} A K\"ahler crepant resolution simply means a crepant resolution that admits  a smooth K\"ahler metric. Neither this K\"ahler metric nor the $\omega$ in the theorem is  required to be  complete, or has any other property beyond being K\"ahler. The reason is that K\"ahler condition suffices for  the local 
$i\partial \overline{\partial}-$lemma \ref{lem strong iddbar} below, and that the gluing region for ansatz is compact and away from the exceptional divisors/central fiber (Lemma \ref{lem Gluing for ansatz} below). Thus any smooth function has bounded $C^{k}-$norms thereon, for any $k$.  We do not need  further detail on the exceptional divisors/central fiber beyond K\"ahlerity.  Existence of a crepant resolution is true for $3-$folds but can possibly fail for higher dimensional orbifolds. 
\end{rmk}

Our examples follow closely  $K3-$surfaces with purely non-symplectic automorphism of  finite order,  including (either directly or heuristically)   \cite{Nikulin, Nikulin0, Nikulin1, AS,  AS4,  AST,  ACV,  Schutt,  Brandhorst,  BH,  MO,  Kondo} etc.   Particularly,  the 39 integers in the following result applies \cite{MO}.
\begin{cor}\label{Cor >2}For any $K3$ surface with a purely non-symplectic automorphism $\sigma$ of finite order, any  K\"ahler crepant resolution 
$\widetilde{\frac{\mathbb{C} \times K3}{\langle  \sigma \rangle}}$  admits an $ALG_{\infty}(\frac{2\pi}{n})$ Ricci flat metric $\omega_{RFALG}$ with Schwartz decay. Because  K\"ahler crepant resolution exists for the $3-$orbifold $\frac{\mathbb{C} \times K3}{\langle  \sigma \rangle}$, for any integer $n$ among the following $8$ prime numbers and $31$ composite numbers,  there is $ALG_{\infty}(\frac{2\pi}{n})$ Ricci flat K\"ahler $3-$fold with Schwartz decay.   \begin{eqnarray}\label{equ composite angles}
& & 2,\ 3,\ 5,\ 7,\ 11,\ 13,\ 17,\ 19,\nonumber
\\& &4,\ 6,\ 8,\ 9,\ 10,\nonumber
\\& & 12,\ 14,\ 15,\ 16,\ 18,\ 20, \nonumber
\\& & 21,\ 22,\ 24,\ 25, \ 26,\ 27,\ 28,\ 30, \nonumber
\\& &32,\ 33,\ 34,\ 36, \ 38, \nonumber
\\& &40,\ 42,\ 44,\ 48,\  \nonumber
\\& &50,\ 54,\ \nonumber
\\& &66. 
\end{eqnarray}
\end{cor}

When the order of the purely non-symplectic automorphism is $2$,  topological data and number of parameters for examples can be calculated fairly easily.
\begin{cor}\label{Cor 2}There is a 1638  parameters family of weakly distinct\footnote{Section \ref{Sect examples}} simply-connected  $ALG_{\infty}
(\pi)$ Ricci flat K\"ahler $3-$folds (with Schwartz decay)  realizing $64$ distinct triples of Betti numbers $(b^{2},b^{3}, b^{4})$.  They live on iso-trivial $K3-$fibrations of which the generic fiber admits a non-symplectic Nikulin involution. 

The examples  have Betti numbers $b^{i}=0$ when $i\geq 5$. 
The pullback of each onto the blowup $Bl_{\mathcal{C}}(\mathbb{C}\times K3)$ along  fixed locus of $\sigma$  is a Ricci flat K\"ahler metric with conic angle $4\pi$ along the exceptional divisor, and also with Schwartz decay to the ALG model on $\mathbb{C}\times K3$ (not the quotient). 
\end{cor}

Instead of from the  $20-$dimensional period domain of all $K3$s,  part of the $1638$ parameters are from sum of dimensions of the corresponding $\rho-$ polarized $K3-$surfaces \cite{AST}.  
ALE/Asymptotically conical Ricci flat   metrics with additional conic  singularity  along a compact divisor with angle $\in (0,  2\pi)$ are constructed by   de Borbon \cite{Borbon} and  de Borbon-Spotti \cite{BS} etc. See Remark \ref{rmk conic Borbon Spotti} below. 

$K3$ surfaces with a non-symplectic Nikulin involution \cite{Nikulin, Nikulin0, Nikulin1} also  appear in the gluing construction for manifolds with $G_{2}-$holonomy by  Kovalev-Lee \cite{KL} and Joyce-Karigiannis \cite{JK}.

 By  $ALG$, we always mean  a positive  rate of decay to the $ALG$ model (see \eqref{equ decay} below and compare it to the concept of weak ALG). 
  We need purely non-symplectic automorphism because if it is not,  the fixed points of the action on $\mathbb{C}\times Y$ does not entirely lie in the fibre $O\times Y$. Namely, $\sigma^{n}$ is a none trivial symplectic automorphism which could fix points on every $Y-$fibre. 
When asymptotic angle is $\pi$,  the compactifying divisors of the examples in Corollary \ref{Cor 2} are smooth rational surfaces.   

In conjunction with the gluing construction of gravitational instantons by Biquard-Minerbe \cite{BM} in complex dimension $2$, it is natural to ask whether it is possible to construct these Ricci-flat metrics by gluing. The difficulty is that the fixed locus consists of  curves, and we do not known whether a singularity model/building block for metric resolution  exists, though we expect   such a perspective metric building block to live on  the rank $2$ orbifold normal bundle of  an  irreducible  point-wisely fixed curve.  We do not known whether a holomorphic orbifold tubular neighborhood exists for  fixed curves, and the expectation is non-existence in most of the cases. To construct exotic Ricci-flat metrics with maximal volume growth on $\mathbb{C}^{n},\ n\geq 3$, in the fundamental $\mathbb{C}\times A_{1}$ case for  \cite{Li, CR, Gabor, Chiu},  schematically speaking,   metric resolution along the $1-$dimensional singularity $\mathbb{C}\times O$ with trivial orbifold normal bundle is obtained.  When the area of the exceptional divisor tends to $0$, it is tempting to expect the ALG Ricci flat metrics to sub-converge in Gromov-Hausdorff topology to the product orbifold metric \eqref{equ model orb product metric}.




 It is tempting to study whether  Corollary  \ref{Cor 2}, \ref{Cor >2}  can possibly yield more examples of generalized Kummer $G_{2}-$manifolds similar to those in \cite{Joyce}, and whether the examples therein can arise as  model of singularity formation of metrics with Ricci curvature bounded from below  at some scale.

 Asymptotically  (real) cylindrical Ricci flat K\"ahler $3-$folds are building blocks for  twisted connected sum in \cite{Kovalev, CHNPDuke, CHNP}. Though far from reach, it is  tempting ask if there is a certain kind of connected sum construction for $G_{2}-$manifolds using  asymptotically \textit{complex cylindrical} Ricci flat K\"ahler $3-$folds in Corollary \ref{Cor >2}, \ref{Cor 2} as building blocks. 

In general dimensions, Tian-Yau \cite{TY1,TY2} utilize Monge-Amper\`e methods to construct complete non-compact Ricci flat K\"ahler metrics.  In \cite{TY1}, they established analytic existence and examples with sub-quadratic volume growth. Also see  \cite{BK1, BK2}.  More recently,  Hein \cite{HeinThesis, HeinJAMES} and  Haskins-Hein-Nordstr\"om \cite{HHN} refined Tian-Yau existence theory.  Based on or related to Monge-Ampere methods,  there have been numerous works on conical and  cylindrical asymptotics etc.  For example,  see  \cite{CHNP,  Kovalev,  Joyce,  CH,  Coevering,  Goto,  Santoro,  Li,   Gabor,  CR,  Chiu,  CL}.  Classifications of  ends,  compactifications,  Torelli-type results, moduli theory,  and metric uniqueness under essentially weaker conditions  are also intensively studied in the setting of  asymptotic cylindrical,  asymptotically  conical, and gravitational instantons. For example,   see  \cite{HeinThesis, HeinJAMES, HSVZ, CC,  SZ0,  SZ, CJL, CV, CVZALG,  Gabor, CS, Chiu, CH2, CH3} etc. The strong uniqueness in \cite{CS} allows slower than quadratic growth on K\"ahler manifolds with maximal volume growth. 

Part of our work is related to a special case  considered in  \cite{Santoro}. Namely,  Santoro \cite{Santoro} constructed $ALF-$Ricci flat metrics on crepant resolutions/isotrivial torus fibration. We focus on ALG case here  with quadratic volume growth, and  provide examples with non-flat $K3-$fibers.  After the first version of the underlying manuscript is posted on arxiv,  the author learned that Johannes Sch\"afer \cite{Schafer} also considered the crepant resolutions (with no $b^{1}(Y)=0$ restriction),  and constructed asymptotic cylindrical steady Ricci solitons thereon via Monge-Ampere methods.   This means the manifolds in Corollary \ref{Cor 2} not only admit ALG Ricci flat metrics with Schwartz decay and quadratic volume growth,  but also steady Ricci soliton structure with linear volume growth.

\subsection{Schwartz Decay}
\begin{prop}\label{prop Schwartz decay for grav instantons} Let
\begin{itemize}\item  $(X^{d},\omega_{0})$ be an iso-trivial ALG K\"ahler manifold, 
\item $k_{0}^{2}$ be the smallest positive eigenvalue of $-\Delta_{Y}$, where $\Delta_{Y}$ is the Laplace Beltrami of the fiber $Y-$component in the ALG model \eqref{equ model orb product metric},  and $k_{0}>0$.
\end{itemize}
Then any smooth solution $\phi$ to the Monge-Ampere equation  \eqref{equ MA on end same volume form} on the end  such that 
 \begin{itemize}\item $\omega_{0}+i\partial\overline{\partial}\phi$ is K\"ahler, 
 
 \item $\phi= O(e^{\tau_{0}r})$ for some $\tau_{0}<k_{0}$,  and
 \item  $i\partial\overline{\partial}\phi=O(\frac{1}{r^{\delta_{0}}})$ for some $\delta_{0}>0$,  
 \end{itemize}
  has Schwartz Decay i.e. for any positive real number $\beta$ and integer $k$, $$|i\partial\overline{\partial}\phi|_{\omega_{0}}=O(\frac{1}{r^{\beta}})\ \textrm{and}\ i\partial\overline{\partial}\phi \in C^{k}_{\beta}(X).$$ Consequently, in complex dimension $2$, the ALG gravitational instantons  in Biquard-Minerbe gluing construction \cite[Theorem 2.3]{BM}   has Schwartz decay. 
\end{prop}


We allow mild exponential growth of  $\phi$. This holds particularly when $\phi$ is bounded. 
The decay is on the K\"ahler form difference $i\partial\overline{\partial}\phi$, not necessarily on the potential function $\phi$. 
It does not explicitly require the constraint. We are not sure whether Hein's ALG reference metric  \cite[Proposition 3.7]{HeinJAMES} on iso-trivial rational elliptic  surfaces are iso-trivial. If they are, among the examples provided by  \cite[Theorem 1.5 (ii)]{HeinJAMES}, the ALG gravitational instantons on iso-trivial rational elliptic fibrations  have Schwartz decay.  

\subsection{Representative examples}
\subsubsection{Asymptotic angle $\pi$ via branched double cover}
Let $\mathcal{C}(10)$ be a (smooth) sextic  in $\mathbb{P}^2$ (of genus $10$). By the construction \cite{BHPV}, there is a smooth K3 surface denoted by  $K3_{\mathcal{C}(10)}$ with a non-symplectic involution $\sigma$, such that there is a covering $K3_{\mathcal{C}(10)}\rightarrow \mathbb{P}^{2}$ branched along $\mathcal{C}(10)$, and $\sigma$ is the covering transformation. As a particular case of Corollary \ref{Cor 2}, Corollary \ref{Cor >2}  implies there is an $ALG_{\infty}(\pi)$ Ricci flat K\"ahler metric on  $\widetilde{\frac{\mathbb{C}\times K3_{\mathcal{C}(10)}}{\langle\sigma\rangle}}$ with Schwartz decay. The exceptional divisor is the ruled surface $\mathbb{P}[O_{\mathcal{C}(10)}\oplus T^{1,0}_{\mathcal{C}(10)}]$. For example,  let $\mathcal{C}(10)$ be   Fermat sextic 
$$\{z^{6}_{0}+az^{6}_{1}+bz^{6}_{2}=0\},\ a\neq 0,\ b\neq 0.$$

The compactifying divisor is simply $\mathbb{P}^{2}$. The Betti numbers $(b^{2}, b^{3},b^{4})$ of $\widetilde{\frac{\mathbb{C}\times K3_{\mathcal{C}(10)}}{\langle\sigma\rangle}}$ equal $(2,20,2)$ respectively. 
\subsubsection{Asymptotic angle $\pi$ via quartic}

 As studied by Sarti  \cite{SartiSlides}, consider the  Fermat quartic:
$$\{X^{4}_{0}+aX^{4}_{1}+bX^{4}_{2}+cX^{4}_{3}=0\},\ a\neq 0,\ b\neq 0,\ c\neq 0$$
with the non-symplectic  involution $\sigma$  defined by
$$\sigma\cdot [X_{0}, X_{1}, X_{2}, X_{3}]= [-X_{0}, X_{1}, X_{2}, X_{3}].$$


The fixed degree $4$ curve $\{aX^{4}_{1}+bX^{4}_{2}+cX^{4}_{3}=0\}\triangleq \mathcal{C}(3)$  lies in the plane $\{X_{0}=0\}$ and has genus $3$. The exceptional divisor is the ruled surface $\mathbb{P}[O_{\mathcal{C}(3)}\oplus T^{1,0}_{\mathcal{C}(3)}].$ Theorem \ref{Thm}  implies there is an $ALG_{\infty}(\pi)$ Ricci flat K\"ahler metric on  $\widetilde{\frac{\mathbb{C}\times K3_{a,b,c}}{\langle\sigma\rangle}}$ with Schwartz decay. The Betti numbers $(b^{2}, b^{3},b^{4})$ of $\widetilde{\frac{\mathbb{C}\times K3_{a,b,c}}{\langle\sigma\rangle}}$ equal $(9,6,2)$ respectively. Via the projection 
 $$ [X_{0}, X_{1}, X_{2}, X_{3}]\longrightarrow [0, X_{1}, X_{2}, X_{3}],$$
  the compactifying divisor 
is the obvious double cover of the plane $$\mathbb{P}^{2}=\{X_{0}=0\}$$ branched along $\mathcal{C}(3).$ 
There is  another involution 
 $$\sigma\cdot [X_{0}, X_{1}, X_{2}, X_{3}]= [-X_{0}, X_{1}, -X_{2}, X_{3}]$$
 which has one more negative sign. It is symplectic because it preserves 
the form $\frac{dx_{0}dx_{1}}{x^{3}_{2}}$ when $X_{2},\ X_{3}\neq 0$.
Symplectic involutions are not  applied to ALG Ricci flat manifolds here i.e. we do not consider asymptotic angle $2\pi$. 

\subsubsection{Asymptotic angle $\frac{2\pi}{3}$ via quartic}

As studied by Artebani-Sarti \cite{AS}, consider the smooth quartic $K3_{quartic, 3}\in \mathbb{P}^{3}$ :
$$\{X^{4}_{0}+X^{4}_{1}+X^{4}_{2}+X_{2}X^{3}_{3}=0\}$$
with the non-symplectic automorphism $\sigma\cdot [X_{0}, X_{1}, X_{2}, X_{3}]= [X_{0}, X_{1}, X_{2},e^{\frac{2\pi i}{3}}X_{3}].$
Theorem \ref{Thm}  implies there is an $ALG_{\infty}(\pi)$ Ricci flat K\"ahler metric on  $\widetilde{\frac{\mathbb{C}\times K3_{quartic,3}}{\langle\sigma\rangle}}$ with Schwartz decay.

\subsubsection{Asymptotic angle $\frac{2\pi}{44}$ via elliptic fibration}

As studied by Kondo \cite{Kondo},  consider the smooth quartic $K3_{44}$ that is the Kodaira-N\'eron model of the Weierstrass fibration:
$$y^{2}=x^{3}+x+t^{11}$$
with the non-symplectic automorphism $\sigma\cdot (x,y,t)\triangleq (\zeta^{22}x,\zeta^{11}y,\zeta^{2}t)$, where $\zeta=e^{\frac{2\pi i}{44}}$. Then $\sigma^{\star}\Omega=\zeta^{13}\Omega$ which implies $\sigma$ is purely non-symplectic. Let $\sigma$ act on $\mathbb{C}$  by multiplication of $\zeta^{31}$ i.e. counter-clockwise rotation in angle $\frac{31 \pi}{22}$. 
Theorem \ref{Thm}  implies there is an $ALG_{\infty}(\pi)$ Ricci flat K\"ahler metric on  $\widetilde{\frac{\mathbb{C}\times K3_{44}}{\langle\sigma\rangle}}$ with Schwartz decay.
 \subsection{Sketch of the proof}
Step 1: The geometric existence is via Tian-Yau existence and Hein decay. The production of the  ansatz can be described by the following diagram. The horizontal  arrow is by local K\"unneth formula/local $i\partial \overline{\partial}-$solvability (Section \ref{sect ddbar}), and gluing Lemma \ref{lem Gluing for ansatz}. The vertical arrow is by a version of Hein and Haskins-Hein-Nordstr\"om technique (Lemma \ref{lem constraint}). 
       \begin{equation}\nonumber  \begin{tikzpicture}
\draw[->,semithick] (-6.1,2) -- (-5.1,2);
\draw[->,semithick] (-3,1.7) -- (-3,0.7);

\node at (-3,2) {Isotrivial ALG metric} ;
\node at (-7.5,2) {K\"ahler metric} ;
\node at (-3,0.5) {Isotrivial ansatz satisfying constraint} ;
\end{tikzpicture}  
\end{equation}
Also see \eqref{equ RFALG potential regarding w -1} and \eqref{equ RFALG potential regarding w} below about the ALG Ricci flat potential \eqref{equ RFALG potential} relative to the initial K\"ahler metric in Theorem \ref{Thm}.

The condition $b^{1}(Y)=0$ is only applied to  the equalities 
\eqref{equ 4 iddar} for the mixed terms. This is crucial for the local $i\partial \overline{\partial}-$Lemma \ref{lem strong iddbar} that yields ALG metric that (strictly) equals the model outside a compact set. Once such an iso-trivial ansatz is constructed, the Schwartz decay does not explicitly require  $b^{1}(Y)=0$. \\

Step 2: The Schwartz decay is via the weighted $L^{2}-$estimates \cite{West} adapted to our ALG manifolds here.  To turn it into (pointwise) weighted Schauder estimate, we need no concentration of $L^{2}-$norm of the Newtonian potential. Namely, because when $R$ is large, an annulus $\{R<r<R+2\}$ on $\mathbb{C}$ has area $O(R)$, that a function $Q_{0}=O(\frac{1}{r^{\beta}})$ implies the $L^{2}-$norm of $Q_{0}$ on the fibred annulus $A(R,R+2)$ is $O(\frac{1}{R^{\beta-\frac{1}{2}}})$. Naive application of the weighted $L^{2}-$estimate \cite{West} on fibered annulus only says that the Newtonian potential $G(Q_{0})$ \eqref{equ decomposition Greens and harmonic} is $O(\frac{1}{r^{\beta-\frac{1}{2}}})$ pointwisely. In other words, this causes ``serious" weight loss  which fails the iteration  for Schwartz decay. To overcome, as opposed to doing Schauder estimate on the fibred annulus, we do it on the fibred sector of angle width comparable to $\frac{1}{R}$ (see Figure \ref{fig annuli}). This is a ``small" part of the fibred annulus that has definite size. The $L^{2}-$norm of the Newtonian potential $G(Q_{0})$ in the fibred sector is bounded by constant comparable to $\frac{1}{R^{\frac{1}{2}}}$ times the $L^{2}-$norm of $Q_{0}$ on the fibred annulus. The no concentration  is because the base of the fibration is $1-$dimensional. This method is expected to fail if the base is of complex dim $2$ or higher.  Eguchi-Hansen metric  on the crepant resolution of $A_{1}-$singularity $\frac{\mathbb{C}^{2}}{\pm Id}$ decays to the flat ALE model in the rate of $O(\frac{1}{r^{4}})$, but not faster. The advantage  of  $\mathbb{S}^{1}$ as the link of the simplest $1-$dimensional cone $\mathbb{C}\setminus O$ is that the eigen functions    $e^{ik\theta}$ has constant norm $1$, therefore has no concentration. This is applied in line 2--3 of \eqref{equ 1 Green func est} below to obtain the factor of $\frac{2}{R}$ which guarantees no loss of weight.  However, on spheres of  dimension $\geq 2$, a sequence of $L^{2}-$normalized eigen-functions indexed by increasing eigenvalues, called Zonal functions, concentrates at  a point \cite{Sogge}. 

The above issue does not happen in the asymptotic cylindrical setting. The reason is in Figure \ref{fig annuli  C cyl}  omitting fiber direction. Also see the more comprehensive Figure \ref{fig annuli}. 
 \begin{figure}[h]
 
  \begin{center}
 
 \begin{tikzpicture}

 \node at (-2,0) {vs};
   \node at (-1,0.79) {$R_{1}$};
      \node at (0.6,0.79) {$R_{1}+2$};
         \node at (2.5,0.79) {$R_{2}$};
            \node at (4.25,0.79) {$R_{2}+2$};
\draw[dashed,color=gray] (-1,-0.5) arc (-90:90:0.2 and 0.5);
	\draw[semithick] (-1,-0.5) -- (0.75,-0.5);
	
	\draw[semithick] (-1,0.5) -- (0.75,0.5);

	\draw[semithick] (-1,-0.5) arc (270:90:0.2 and 0.5);
	\draw[dashed,color=gray] (0.75,-0.5) arc (-90:90:0.2 and 0.5);
		\draw[semithick] (0.75,-0.5) arc (270:90:0.2 and 0.5);

\draw[dashed,color=gray] (2.5,-0.5) arc (-90:90:0.2 and 0.5);
	\draw[semithick] (2.5,-0.5) -- (4.25,-0.5);
	
	\draw[semithick] (2.5,0.5) -- (4.25,0.5);

	\draw[semithick] (2.5,-0.5) arc (270:90:0.2 and 0.5);
	\draw[dashed,color=gray] (4.25,-0.5) arc (-90:90:0.2 and 0.5);
		\draw[semithick] (4.25,-0.5) arc (270:90:0.2 and 0.5);
	
		 \draw (-4,0) circle (1);
     \draw (-4,0) circle (0.7);
    
  	 \draw (-7,0) circle (0.6);
     \draw (-7,0) circle (0.3);
  
\end{tikzpicture} 

\end{center} \caption{Two annulus of the same radius. In the complex plane $\mathbb{C}$, the shape/area depends on their inner radius. But in the standard cylinder $\mathbb{R}\times \mathbb{S}^{1}$, they look the same.}\label{fig annuli  C cyl}
\end{figure}
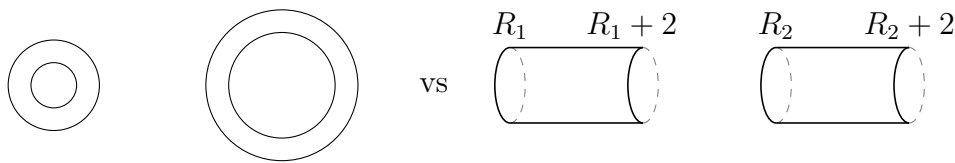

       Organization: Most definitions and background can be found in Section \ref{sect Setup} and the introduction.  In Section \ref{sect Schwartz}, assuming an ansatz that equals a Ricci flat ALG model outside a compact set (of which the construction will be deferred to Section  \ref{sect ansatz}), we utilize the weighted $L^{2}-$estimates in \cite{West} to prove Schwartz decay. In Section \ref{sect ddbar}, on a K\"ahler $Y-$fibration over the punched $1-$dimensional disk such that $b^{1}(Y)=0$, we show that a closed $(1,1)-$form is $i\partial\overline{\partial}-$cohomologous to a fixed $(1,1)-$form on $Y$ (independent of the fiber). With this ``local K\"unneth for $H^{1,1}$",   in Section \ref{sect ansatz}, we construct  ansatz that equals the (Ricci flat) ALG model on the end and satisfies the constraint in Tian-Yau solvability \cite[Theorem 1.1]{TY1} and Hein's version \cite[Proposition 4.1]{HeinThesis}. The desired Ricci flat ALG metric is therefore constructed, and we review Hein's polynomial decay estimate. In Section \ref{sect orbifold metric}, as fairly explicit examples of initial K\"ahler metrics required by Theorem  \ref{Thm},  we describe the  crepant resolution of the orbifold and  the  Bergman metric as a consequence of Kodaira-Baily embedding \cite{Baily}.  In Section \ref{sect Topology}, we calculate the fundamental group and Betti numbers of the examples, and count the parameters of the family.   \\
       
       \textbf{Acknowledgement}:  The author is grateful to Simon Donaldson, Song Sun, Yunfeng Jiang, Yongqiang Liu,  Craig van Coevering,   Simon Brandhorst,   Ruobing Zhang,  and Chi Li  for valuable discussions.

\section{Setup \label{sect Setup}}
\subsection{For the main theorem}
Manifolds are assumed  connected throughout unless otherwise indicated. For example, the exceptional divisor of the resolution and fixed locus of the automorphism $\sigma$ could be reducible. 
\begin{Def}\label{Def WHF}
\begin{enumerate}
\item  An admissible  K\"ahler configuration means a closed K\"ahler manifold $Y$ of complex dimension $d-1$, with an automorphism $\sigma$  of finite order $n$.

 \item  A non-compact complete K\"ahler $d-$fold $(X^{d}, \omega)$    
is said to be weak $\beta-$ALG if there is an admissible  K\"ahler configuration $(Y,\omega_{Y},\sigma)$,  compact sets $K\subset X$, $K^{\prime} \subset  \frac{\mathbb{C}\times Y}{ \langle \sigma \rangle}$,  and  a diffeomorphism 
\begin{equation}\label{equ coordinate near infty}\Phi:\ \frac{\mathbb{C}\times  Y}{ \langle \sigma \rangle} \setminus K^{\prime}\rightarrow X\setminus K\end{equation} such that $\Phi^{-1}$ composed with (restriction of) the  projection $$\rho_{1}:\ \ \frac{\mathbb{C}\times  Y}{ \langle \sigma \rangle}\longrightarrow \frac{\mathbb{C}}{\langle \sigma \rangle}$$
 is holomorphic from $X\setminus K\longrightarrow Range(\rho_{1}\circ \Phi^{-1})\subset \frac{\mathbb{C}}{\langle \sigma \rangle}$, and 
\begin{equation}\label{equ model metric quasi isometry general beta}\frac{\frac{i}{2}\lambda^{2}|z|^{2\beta-2}dzd\overline{z}+\omega_{Y}}{C_{X,g}}\leq \Phi^{\star}(\omega)\leq C_{X,g}(\frac{i}{2}\lambda^{2}|z|^{2\beta-2}dzd\overline{z}+\omega_{Y}),
\end{equation}

for some positive number $C_{X,g}$.  $\omega_{Y}$ is a $\sigma-$invariant K\"ahler form on $Y$, and $\lambda^{2}$ is  a positive real number called the radial scaling factor. The action of $\sigma$ on $\mathbb{C}$ is multiplication by a primitive $n-$th unit root. When $\beta=1$, we call it weak ALG. A weak ALG manifold is said to be $ALG_{\delta}(\frac{2\pi}{n})$ for some $\delta>0$, or simply ALG (sometimes   $\delta$ and/or $\frac{2\pi}{n}$ are suppressed) if for any natural number $k$, 
\begin{equation}\label{equ decay}|\nabla^{k}[\Phi^{\star}(\omega)-(\frac{i}{2}\lambda^{2} dzd\overline{z}+\omega_{Y})]|\leq C_{k,\delta}r^{-\delta},
\end{equation}
for some $C_{k,\delta}$ independent of $r\triangleq |z|$.  We call 
\begin{equation}\label{equ model orb product metric}\omega_{model}\triangleq \frac{i}{2}\lambda^{2} dzd\overline{z}+\omega_{Y}\ \ \textrm{the ALG model},\end{equation}  under which the norms/covariant derivatives in \eqref{equ decay} are defined. We say such a manifold has Schwartz decay if it is $ALG_{\delta}(\frac{2\pi}{n})$ for any $\delta>0$, in which case we  call it $ALG_{\infty}(\frac{2\pi}{n})$ as well.  The compact sets $K$ and $K^{\prime}$ are not necessarily required to be uniform in $\delta$. 
 
Unlike \cite{HeinThesis, CC}, we do not require faster decay for higher derivatives  formally in \eqref{equ decay}. But there is no essential difference because we prove Schwartz decay anyway. 
 \item A weak ALG manifold (without specifying metric) is called an  isotrivial ALG manifold if the $\Phi$ in \eqref{equ coordinate near infty} is a biholomorphism. An  ALG K\"ahler metric on an iso-trivial  ALG manifold  is called iso-trivial if it (strictly) equals a model \eqref{equ model orb product metric} away from a compact set in $X$ under the holomorphic coordinate $\Phi$ near $\infty$.  
 
 The following implication on metrics holds.
$$Isotrivial\  ALG \Longrightarrow ALG_{\infty}\Longrightarrow ALG\Longrightarrow \textrm{weak}\ ALG.$$
 

\item Continuing Definition \ref{Def Y sigma before thm} and Theorem \ref{Thm}, we usually  denote the K\"ahler crepant resolution $\widetilde{\frac{\mathbb{C} \times Y}{\langle  \sigma \rangle}}
$ by $X$. 

By crepant resolution, we include the property that the pullback $\Phi^{-1,\star}(dz\wedge \Omega_{Y})$ of the invariant holomorphic volume form extends to a now-where vanishing holomorphic $(d,0)-$form on $X$. This automatically holds when complex dimension is $3$ under the setting of Corollary \ref{Cor >2},  since the discrepancy is $0$ (see \cite[Lemma 6, extension in the proof]{Roan} and \cite[Example 1.9, the application of Theorem D]{CHNP}). 

The compactification $\overline{X}$ is 
$$\widetilde{\frac{\mathbb{P}^{1} \times Y}{\langle  \sigma \rangle}},$$
where $\sigma$ still acts by $\frac{1}{\zeta_{n}}$  on 
$\mathbb{P}^{1}=\mathbb{C}\cup \{\infty\}$,  and the resolution is only along $\{0\} \times Fix(\sigma)$, but not along $\{\infty\}\times Fix(\sigma)$.  $u=\frac{1}{z}$ defines the   coordinate $\Phi$ near $\infty$, and the compactifying divisor is $\frac{\{0_{u}\}\times Y}{\langle\sigma \rangle}$ (also denoted by $\frac{\{\infty\}\times Y}{\langle\sigma \rangle}$ in terms of $z$). 
\end{enumerate}

\end{Def}

For any $\infty\geq  b>a>0$, let $A(a,b)$ denote the (open) annulus $$(a,b)\times \mathbb{S}^{1}\times Y\subseteq \mathbb{C}^{\star}\times Y$$   When $b=\infty$, we also call it ``the end".  Let $\psi_{\mu}$ denote the biholomorphism 
  from $A(r_{0},\infty)$ to $A(\mu r_{0},\infty)$ defined by the radial scaling
  $$\psi_{\mu}(z,y)=(\mu z,y).$$
  When $r_{0}$ is large enough, the model metric  \eqref{equ model metric quasi isometry general beta} becomes the normalized
  \begin{equation}\label{equ normalized model}\frac{i}{2}|z|^{2\beta-2}dzd\overline{z}+\omega_{Y}\ \textrm{under the coordinate}\ \psi_{\lambda^{-\frac{1}{\beta}}}\circ \Phi.
  \end{equation}
  The proof of Theorem \ref{Thm} (including the lemmas, propositions etc it requires) are  under the normalizing chart $\psi_{\lambda^{-\frac{1}{\beta}}}\circ \Phi$ unless otherwise specified.  The $\lambda$  adds one more parameter to the family of examples.

\subsection{For Monge-Ampere equation}
The  identity
\begin{equation}
\omega_{\mathbb{C}^{d}}^{d}=\nu_{d}\Omega_{\mathbb{C}^{d}}\wedge \overline{\Omega}_{\mathbb{C}^{d}},\ \nu_{d}\triangleq \left\{\begin{array}{c}
\frac{d!}{2^{d}}\ \textrm{when}\ d\ \textrm{is even},\\
\frac{d!}{2^{d}}\cdot i\ \textrm{when}\ d\ \textrm{is odd}
\end{array}\right.
\end{equation}
holds for the standard K\"ahler metric and   holomorphic volume forms on the Euclidean space:
$$\omega_{\mathbb{C}^{d}}=\frac{i}{2}(dz_{1}d\overline{z}_{1}+...+dz_{d}d\overline{z}_{d}),\ \ \Omega_{\mathbb{C}^{d}}=dz_{1}dz_{2}\wedge...\wedge dz_{d}.$$

On an iso-trivial Calabi-Yau fibration $\widetilde{\frac{\mathbb{C} \times Y}{\langle  \sigma \rangle}}$ as Definition \ref{Def WHF},  define the Ricci potential/volume density
\begin{equation}\label{equ volume density}f\triangleq \log(\frac{\nu_{d}\Omega\wedge \bar{\Omega}}{\omega^{d}_{0}}).\end{equation} 

 It suffices to solve the target (elliptic) Monge-Ampere equation 
\begin{equation}\label{equ Target MA}
(\omega_{0}+i\partial \overline{\partial}\phi)^{d}=e^{f}\omega^{d}_{0},
\end{equation}
which is equivalent to the volume form equation:
\begin{equation}
\omega_{\phi}^{d}=\nu_{d}\Omega\wedge \bar{\Omega},\ \textrm{where}\ \omega_{\phi}\triangleq \omega_{0}+i\partial \overline{\partial}\phi.
\end{equation}

Tian-Yau existence \cite[Theorem 1.1]{TY1} and Hein's version \cite[Proposition 4.1]{HeinThesis} require an ansatz  that obeys the constraint
\begin{equation}\label{equ constraint}
\int_{X}(\omega_{0}^{d}-\nu_{d}\Omega\wedge\overline{\Omega})=0. 
\end{equation}
As a premise we need the top form $\omega_{0}^{d}-\nu_{d}\Omega\wedge\overline{\Omega}$ to be integrable over $X$. \\
\subsection{Other terms and conventions}

\textit{Terms}:\begin{itemize}\item  Let $orb$ denote $\frac{\mathbb{C}\times  Y}{ \langle \sigma \rangle}$ or the special case $\frac{\mathbb{C}\times  K3}{ \langle \sigma \rangle}$ of interest, and call it ``the orbifold".

\item Sometimes or always, by ``K\"ahler metric", we mean the K\"ahler form.

  \item Let $R_{0}>>1$ be large enough so all the subsequent desired operations are valid. We extend the model radius $r$, from $A(R_{0},\infty)$ to be valued in $[0,1]$ on the whole $X$ and $\equiv 0$  in $X\setminus A(R_{0}-1000,\infty)$.  
  
  \item Let $\Delta^{\star}$ or $\Delta^{\star}_{u}$ be the punched disk  $\{u\in \mathbb{C}|\ 0<|u|<\rho_{0}\}$, $u\triangleq \frac{1}{z}$ is the base coordinate near $\infty$.  The radius $\rho_{0}$ is suppressed and is usually taken large enough depending on context.  See for example section \ref{sect ddbar}. 
\end{itemize}

\textit{Convention for constants}:  Constants $C$ (might be different at different places) and radius $R_{i}$ are  positive  and depend on the underlying manifold/orbifold which is fixed and  clear from context. They might depend on  existing/preceding reference metrics, but never on the reference metric we are constructing,  or the Ricci flat metric/K\"ahler potential  we solve for in the current step. Sometimes we add subscript to indicate dependence, and/or use other symbols. 

\section{Schwartz decay assuming an isotrivial ansatz\label{sect Schwartz}}

\subsection{Preliminary and idea}
On the end of the resolution $X=\widetilde{\frac{\mathbb{C}\times Y}{\langle \sigma \rangle}}$, we study the metrics and functions  in the orbifold chart upstairs i.e. we view them as being $\sigma-$invariant on the fibred annulus $A(\rho,\infty)=[\mathbb{C}\setminus D(\rho)]\times Y$ where $D(\rho)\subset \mathbb{C}$ is a disk of radius $\rho$ centered at the origin.

\textit{Iso-triviality of the ansatz $\omega_{0}$ constructed at the end of Section \ref{sect ddbar} implies the Ricci potential/volume density $f$ \eqref{equ volume density}  is compactly supported}. Away from the support of $f$,   the target equation \eqref{equ Target MA} reads
\begin{equation}\label{equ MA on end same volume form}
\omega_{\phi}^{d}=\omega_{model}^{d},\ \textrm{where}\ \omega_{\phi}\triangleq \omega_{model}+i\partial \overline{\partial}\phi. 
\end{equation}
The equality \eqref{equ Kovalev form of MA} below becomes 
 \begin{equation}\label{equ Kovalev form of MA}
 \frac{i\partial \overline{\partial}\phi\wedge[\omega_{\phi}^{d-1}+...+\omega_{model}^{j}\wedge \omega_{\phi}^{d-1-j}+...+\omega_{model}^{d-1}]}{\omega_{model}^{d}}=0.
 \end{equation}
Foiling the left side yields 
\begin{equation}\label{equ Laplace}
\Delta_{model}\phi=Q(i\partial \overline{\partial}\phi),
\end{equation}
where  $Q$ is quadratic  in $i\partial \overline{\partial}\phi$, and $\Delta_{model}$ is the Laplace Beltrami of the ALG model. Then we bootstrap the decay rate using the main theorem in  \cite{West}   $$i\partial \overline{\partial}\phi= O(\frac{1}{r^{\beta}})\Longrightarrow Q(i\partial \overline{\partial}\phi)=O(\frac{1}{r^{2\beta}})\Longrightarrow i\partial \overline{\partial}\phi= O(\frac{1}{r^{2\beta}}).$$

Successive application of the above yields that $|i\partial \overline{\partial}\phi|$ has faster than polynomial decay i.e. $|i\partial \overline{\partial}\phi|=o(\frac{1}{r^{p}})$ for any $p>0$.

\subsection{Proof \label{sect proof}}
We say a continuous function $f$ on a cone end $A(\varrho,\infty)$ has vanishing fiber-wise integrals if for any $z$ with $|z|>\varrho$ (where $\varrho>0$ is fixed), the integral of $f$ over $z\times Y$ vanishes. We decompose  equation \eqref{equ Laplace} according to $Ker\Delta_{Y}$ and the orthogonal complement:
 \begin{equation}\label{equ decomposition}
 \phi=\phi_{0}+\underline{\phi},\ Q=Q_{0}+\underline{Q}
 \end{equation}
 where $\underline{\phi},\ \underline{Q}$ are pullback functions from $\mathbb{C}$,  and $\phi_{0},\ Q_{0}$ have vanishing fiber-wise integrals.  This means 
 \begin{equation}
 \underline{\phi}=\frac{\int_{Y}\phi \frac{\omega^{d-1}_{Y}}{(d-1)!}}{\int_{Y}1 \frac{\omega^{d-1}_{Y}}{(d-1)!}}. \ \ \textrm{Moreover}, \ i\partial \overline{\partial}\underline{\phi}=i\partial_{\mathbb{C}} \overline{\partial}_{\mathbb{C}}\underline{\phi}=\frac{\int_{Y}(i\partial_{\mathbb{C}} \overline{\partial}_{\mathbb{C}}\phi) \frac{\omega^{d-1}_{Y}}{(d-1)!}}{\int_{Y}1 \frac{\omega^{d-1}_{Y}}{(d-1)!}}.\ 
 \end{equation}
Therefore decay (estimates) of $i\partial \overline{\partial}\phi$ is equivalent to  that of both  $i\partial \overline{\partial}\underline{\phi}$ and $i\partial \overline{\partial}\phi_{0}$ as 
\begin{equation}\label{equ ddbar of phi est equiv to both parts}
i\partial \overline{\partial}\phi_{0}=i\partial \overline{\partial}\phi-\frac{\int_{Y}(i\partial_{\mathbb{C}} \overline{\partial}_{\mathbb{C}}\phi) \frac{\omega^{d-1}_{Y}}{(d-1)!}}{\int_{Y}1 \frac{\omega^{d-1}_{Y}}{(d-1)!}}.
\end{equation}
 We do not decompose the $i\partial \overline{\partial} \phi$ in $Q$,  but directly decompose $Q$ as a whole. Because $\Delta_{Y}\underline{\phi}=0$,  \eqref{equ Laplace}  is decomposed into the following two equations 
 \begin{equation}\label{equ decomposed}
 \Delta \phi_{0}=Q_{0},\ \Delta_{\mathbb{C}} \underline{\phi}=\underline{Q}.
 \end{equation}
 The point is that $\Delta_{\mathbb{C}}=4\frac{\partial^{2}}{\partial z \partial \overline{z}}$. Hence the second equation in \eqref{equ decomposed} says 
  \begin{equation}\label{equ identify ddc with laplace}
i\partial \overline{\partial}\underline{\phi} =\frac{\underline{Q}}{4}i dz d \overline{z}\ \ \ \textrm{automatically decays because}\ \underline{Q}\ \textrm{does}. 
 \end{equation}
The main theorem in \cite{West} gives a Green's function for the first equation in \eqref{equ decomposed} with certain bounds. This yields the following weighted  Schauder (pointwise) estimate. 

\begin{lem}\label{lem weighted Schauder}On the end,  for any real number $\beta$, natural number $k\geq 1000(1000^{d}+1)$,   non-negative real number $\tau_{0}$ as Proposition \ref{prop Schwartz decay for grav instantons},   and  $R >1000$,  there exists a constant $C$ such that the following  estimate holds whenever $\phi_{0}$  have vanishing fiber-wise integral, and is   $C^{k+2,\frac{1}{2}}$ in  $A(R-1,\infty)$. 
\begin{equation}\label{equ lem 4.1}
|i\partial \overline{\partial}\phi_{0}|_{C_{\beta}^{k,\frac{1}{2}}[A(R-1,\infty)]}\leq C\{|e^{-\tau_{0}r}\phi_{0}|_{C^{0}[A(R-2,\infty)]}+|\Delta \phi_{0}|_{C_{\beta}^{k,\frac{1}{2}}[A(R-2,\infty)]}\}.
\end{equation}
\end{lem}
By smoothness of Tian-Yau solution  \cite[Theorem 1.1, last sentence]{TY1}, we only need a priori estimate. Combining the trivial identification \eqref{equ identify ddc with laplace}, we find for any large $R_{1}$ that
\begin{equation}
|i\partial \overline{\partial}\phi|_{C_{\beta}^{k,\frac{1}{2}}[A(R_{1}+1,\infty)]}\leq C_{\beta}\{|e^{-\tau_{0}r}\phi|_{C^{0}[A(R_{1},\infty)]}+|Q|_{C_{\beta}^{k,\frac{1}{2}}[A(R_{1},\infty)]}\}.
\end{equation}

Starting from the initial decay Lemma \ref{equ initial decay},  induction yields 
\begin{eqnarray*}& & |i\partial \overline{\partial}\phi|_{C_{2^{j}\beta_{0}}^{k,\frac{1}{2}}[A(R_{1}+j,\infty)]}
\\& \leq & C_{2^{j}\beta_{0}}\cdot  C_{2^{j-1}\beta_{0}}\cdot \cdot \cdot \cdot C_{2\beta_{0}}(|e^{-\tau_{0}r} \phi|_{C^{0}[A(R_{1},\infty)]}+|i\partial \overline{\partial}\phi|_{C_{\beta_{0}}^{k,\frac{1}{2}}[A(R_{1},\infty)]})
\end{eqnarray*}
for any non-negative integer  $j$. The $k$ is fixed, large, and arbitrary during the induction.  Still under the setting of Proposition \ref{prop Schwartz decay for grav instantons},  standard argument\footnote{Section \ref{section weighted Schauder norm}} incorporating the no-concentration estimate \eqref{equ 1 Green func est} shows

\begin{equation}\label{equ higher order initial decay}|i\partial \overline{\partial}\phi|_{C_{\beta_{0}}^{k,\frac{1}{2}}[A(R_{1},\infty)]}\leq C[|e^{-\tau_{0}r} \phi|_{C^{0}[A(R_{1}-10,\infty)]}+|i\partial \overline{\partial}\phi|_{C_{\beta_{0}}^{0}[A(R_{1}-10,\infty)]}].
\end{equation}

This means the $C^{0}-$ decay conditions in Proposition \ref{prop Schwartz decay for grav instantons} imply the higher order   decay in rate $\beta_{0}=\delta_{0}>0$ to initiate the iteration.  The proof for  Schwartz decay is complete. 

\subsection{Weighted Schauder estimate and Mazja-Plamenevskii technique}

\begin{proof}[Proof of Lemma \ref{lem weighted Schauder}:] Step 1: Preliminary. 

Let $Q_{0}$ denote $\Delta\phi_{0}$, and $G$ denote the Green's function constructed in \cite{West}. Here we extend $Q_{0}$ to be $0$ outside the end $A(R,\infty)$.  This does not affect $L^{2}-$integrability. Then we  apply the $G$ in \cite{West} and restrict $G\circ Q_{0}$ on the end.  We find  the decomposition
\begin{equation}\label{equ decomposition Greens and harmonic}\phi_{0}=G\circ Q_{0}+h\  \textrm{where}\ \ h\ \ \textrm{is harmonic}.
\end{equation}  We estimate them separately.

Let $p=(R,\theta_{0},x)\in (0,\infty)\times \mathbb{S}^{1}\times Y$ be an arbitrary point such that $R>1000$. The regularity we assumed ensures that  the summation of Fourier-Series equals the function itself i.e. $$Q_{0}=\Sigma_{k\in \mathbb{Z},\mu\in Spec_{mul} \Delta_{Y}}(Q_{0})_{k,\mu}e^{ik\theta}\psi_{\mu},$$
where $Spec_{mul} \Delta_{Y}$ is the spectrum (counted with multiplicity)  of the Laplace Beltrami on $Y$ with respect to the invariant K\"ahler metric $\omega_{Y}$, and $\psi_{\mu}$ is the corresponding eigen-function.  The Green's function constructed in \cite[Theorem 1.2]{West} respects the Fourier decomposition i.e. for any $k,\mu$ as indicated, there exists a Green's function $G_{k,\mu}$ on functions in $r$ only,  such that  $$G(Q_{0})=\Sigma_{k\in \mathbb{Z},\mu\in Spec_{mul}}G_{k,\mu}[(Q_{0})_{k,\mu}]e^{ik\theta}\psi_{\mu}.$$

Step 2: Estimate  for  the Newtonian potential $G(Q_{0})$. 

The estimates in \cite{West} are all about $L^{2}$/Sobolev norms. Here we translate them to pointwise estimates. 

We apply the usual interior Schauder estimate in a small scale with respect to the convex, harmonic radius of $Y$, and Tian-Yau ``radius".  The idea is to integrate the Mazja-Plamenevsky estimate \cite[Theorem 1.2.(4)]{West} on 
a sector  over  the ``small" portion $(\theta_{0}-\frac{1}{R},\theta_{0}+\frac{1}{R})$ of the circle (Figure \ref{fig annuli}). 
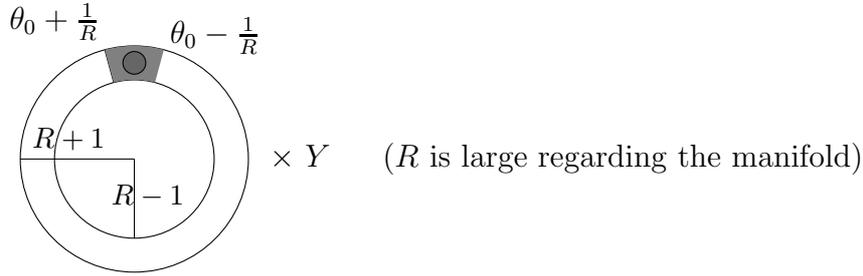
\begin{figure}[h]
\begin{center}
\begin{tikzpicture}[scale=1.5]
  \draw (0,0) circle (1);
     \draw (0,0) circle (0.7);
     \draw (105:1)--(105:0.7);
       \node  at (120:1.4) {$\theta_{0}+\frac{1}{R}$};
        \node  at (57:1.3) {$\theta_{0}-\frac{1}{R}$};
       \coordinate [label=right :$\ \times\ Y\ \ \ \ \ (R$ is large regarding the manifold)] (E) at (0:1);
    \path[fill=gray] (105:1)--(105:0.7) arc (105: 75: 0.7)--(75: 0.7) -- (75:1) arc (75: 105: 1)--cycle;
 \draw (0,0)--(270:0.7);
 \node at (290:0.35) {\small{$R-1$}};
  \draw (0,0)--(180:1);
 \node at (163:0.6) {\small{$R+1$}};
    \draw[fill=gray!120] (90:0.85) circle (0.1);

 \end{tikzpicture} \end{center}\caption{gray region$\times Y$ is a sector-shaped annulus}\label{fig annuli}
 \end{figure}

The weighted $L^{2}-$norm on the sector-shaped annulus is estimated as follows, where the  measure  is the standard $rdrd\theta dY$.
\begin{eqnarray}& &|r^{\beta}G(Q_{0})|^{2}_{L^{2}[(R-1,R+1)\times (\theta_{0}-\frac{1}{R},\theta_{0}+\frac{1}{R})\times Y]}\nonumber
\\&=&\Sigma_{k\neq 0,\mu}\int_{Y}\int^{\theta_{0}+\frac{1}{R}}_{\theta_{0}-\frac{1}{R}}\int_{R-1}^{R+1}|G_{k,\mu}[(Q_{0})_{k,\mu}]|^{2} |e^{ik\theta}|^{2}\cdot  |\psi_{\mu}|^{2}r^{2\beta} rdrd\theta dY\nonumber
\\&=&\frac{2}{R}\Sigma_{k\neq 0,\mu}\int_{R-1}^{R+1}|G_{k,\mu}[(Q_{0})_{k,\mu}]|^{2}r^{2\beta}rdr
=\frac{1}{\pi R} |r^{\beta}G(Q_{0})|^{2}_{L^{2}[A(R-1,R+1)]}\nonumber
\\& \leq & C \sup_{\varrho\geq R}\frac{1}{\varrho} |r^{\beta}G(Q_{0})|^{2}_{L^{2}[A(\varrho-1,\varrho+1)]}
\leq C\sup_{\varrho\geq R} \varrho^{2\beta}|Q_{0}|^{2}_{C^{0}[A(\varrho-1,\varrho+1)]}. \label{equ 1 Green func est}
\end{eqnarray}
The reason  we lose the $\varrho-$factor in line $4$ is that the area of the fibred annulus is comparable to $R$. In line 3,  we must sum up   before  applying  Mazja-Plamenevski estimate  \cite[Theorem 1.2, (4)]{West} to obtain line 4, where the $L^{2}-$norms  are on the full fibred annulus, not only the sector.

The ball $B_{p}(\frac{1}{10})$ (deeper gray ball in Figure \ref{fig annuli}) is contained in the interior of the sector $(a,b)\times (\theta_{0},\theta_{1})\times Y$,  using the $L^{2}$ norm as the lowest order term in the Schauder interior  estimate, we find 
\begin{eqnarray}& &R^{\beta}|G(Q_{0})|_{C^{k+2,\frac{1}{2}}[B_{p}(\frac{1}{10})]}\nonumber
\\&\leq &  CR^{\beta}\{|G(Q_{0})|_{L^{2}[(a,b)\times (\theta_{0},\theta_{1})\times Y]}+|Q_{0}|_{C^{k,\frac{1}{2}}[(a,b)\times (\theta_{0},\theta_{1})\times Y]}\}\nonumber
\\&\leq & C\cdot \sup_{\varrho \geq R} R^{\beta}|Q_{0}|_{C^{k,\frac{1}{2}}[A(\varrho-1,\varrho+1)]}\ \ \ \label{equ 1 Schwarts strip} (\textrm{we applied}\  \eqref{equ 1 Green func est})
\\&\leq & C|Q_{0}|_{C_{\beta}^{k,\frac{1}{2}}[A(\varrho-1,\infty)]}.\nonumber
\end{eqnarray}

Step 3: Estimate for the harmonic part. Using the spectrum gap, let $\tau\triangleq\frac{\tau_{0}+k_{0}}{2}\in (\tau_{0}, k_{0})$ in Proposition \ref{prop Schwartz decay for grav instantons}, we find 
\begin{eqnarray}& & R^{\beta}|h|_{C^{k+2,\frac{1}{2}}[B_{p}(\frac{1}{10})]}\nonumber
\leq C|r^{\beta}h|_{L^{2}[B_{p}(\frac{1}{10})]}  \leq C|r^{\beta}e^{-\tau r}h|_{L^{2}[A(R-1,\infty)]}
\\& \leq & C|e^{-\tau_{0} r}h|_{C^{0}[A(R-1,\infty)]} \nonumber
 \leq  C[|e^{-\tau_{0} r} \phi_{0}|_{C^{0}[A(R-1,\infty)]}+  |e^{-\tau_{0} r} G\circ Q_{0}|_{C^{0}[A(R-1,\infty)]}]\nonumber
\\& \leq & C[|e^{-\tau_{0} r} \phi_{0}|_{C^{0}[A(R-2,\infty)]}+|Q_{0}|_{C_{\beta}^{\frac{1}{2}}[A(R-2,\infty)]}].\label{equ 2 Schwarts strip}
\end{eqnarray}

Step 4: Summing \eqref{equ 1 Schwarts strip} and \eqref{equ 2 Schwarts strip} up for the decomposition \eqref{equ decomposition},  we find 
\begin{eqnarray}& &|i\partial \overline{\partial}\phi_{0}|_{C_{\beta}^{k,\frac{1}{2}}[A(R-1,\infty)]}\leq  |i\partial \overline{\partial} [G (Q_{0})] |_{C_{\beta}^{k,\frac{1}{2}}[A(R-1,\infty)]}+|i\partial \overline{\partial}h |_{C_{\beta}^{k,\frac{1}{2}}[A(R-1,\infty)]}\nonumber
\\& \leq &  C\{|e^{-\tau_{0} r}\phi_{0}|_{C^{0}[A(R-2,\infty)]}+|Q_{0}|_{C_{\beta}^{k,\frac{1}{2}}[A(R-2,\infty)]}\}. \nonumber
\end{eqnarray}
The proof  is complete. \end{proof}
\section{Local $i\partial \overline{\partial}-$lemma  on the ALG end \label{sect ddbar}}

\begin{lem}\label{lem strong iddbar} Let $Y$ be a compact K\"ahler manifold, and $[ \ \cdot \ ]$ denote the cohomology class of a $(1,1)-$form   in $H^{1,1}[Y,\mathbb{C}]$.  Let $\omega$ be a closed $(1,1)-$form on the trivial fibration $\Delta^{\star}\times Y$.  Then the restricted $(1,1)-$cohomology class  $[\omega(u)]$ on $Y$ is independent of $u$. Suppose  $b^{1}(Y)=0$ additionally. Then for any $u_{0}\in \Delta^{\star}$, there is a smooth function $\phi_{0}$ on $\Delta^{\star}\times Y$  such that 
 \begin{equation}\label{equ form of w iddbar}
 \omega=i\partial \overline{\partial}\phi_{0}+\omega_{Y,0},
 \end{equation}
 where $\omega_{Y,0}\triangleq \omega_{Y}(u_{0})$ is the restriction of $\omega$ onto the fiber $u_{0}\times Y$, viewed as a $\Delta^{\star}-$independent form.  

  If $\sigma$ is a non-id  automorphism of $Y$ of  finite order $ord_{\sigma}$, let $\sigma$ act on $\Delta^{\star}$ via multiplication by a primitive $ord_{\sigma}-$th unit root.     If $\omega$ is real and/or  $\sigma-$invariant,  then  $\phi_{0}$ can be taken real and/or $\sigma-$invariant as well.
  

   \end{lem}

 Our consideration here is different from the $i\partial \overline{\partial}-$solvability for elliptic fibrations  \cite{HeinThesis, HeinJAMES, CVZALG, GW},  because Betti number of a torus does not vanish,  but $b^{1}(Y)=0$ is an important condition in Lemma \ref{lem strong iddbar}.    Standard Hodge decomposition obviously ``fails" in  this local ``toy" case i.e.  $b^{1}[\Delta^{\star}\times Y]=1$, but our point is that the K\"unneth formula for $H^{1,1}$ is still true:
   $$H^{1,1}[\Delta^{\star}\times Y,\mathbb{C}]=H^{1,1}[Y,\mathbb{C}],\ H^{1,1}[\Delta^{\star}\times Y,\mathbb{R}]=H^{1,1}[Y,\mathbb{R}].$$
We recall the pedestrian language that  $H^{1,1}$ means the vector space of all $(1,1)-$classes i.e. cohomology classes in $H^{2}$ represented by $(1,1)-$forms (either complex or real, depending on the co-efficient symbol). 

  \begin{proof}Let  $\rho_{1},\ \rho_{2}$ be projections from  $\Delta^{\star}\times Y$ to the 
$\Delta^{\star}$ and $Y$ components, respectively. Under the decomposition
\begin{eqnarray*}& &\wedge^{1,1}[\Delta^{\star}\times Y]
\\&=&\rho_{1}^{\star}\wedge^{1,1}[\Delta^{\star}] \oplus (\rho_{1}^{\star}\wedge^{1,0}[\Delta^{\star}]\wedge  \rho_{2}^{\star}\wedge^{0,1}[Y])\oplus (\rho_{1}^{\star}\wedge^{0,1}[\Delta^{\star}]\wedge  \rho_{2}^{\star}\wedge^{1,0}[Y])
\\& & \oplus  \rho_{2}^{\star}\wedge^{1,1} [Y],
\end{eqnarray*}
we write   \begin{equation}\label{equ -1 iddbar}\omega=aidu d\overline{u}+du\wedge F^{0,1}+d\overline{u}\wedge F^{1,0}+\omega_{Y}(u)\end{equation}
according to the order of the $4-$components  where $a$ is a function, and \\$\ F^{1,0}, F^{0,1},\ \omega_{Y}(u)$ are sections of $\rho_{2}^{\star}\Omega^{1,0}[Y],\ \rho_{2}^{\star}\Omega^{0,1}[Y],\ \rho_{2}^{\star}\Omega^{1,1}[Y]$ respectively. \\

Step 1: Independence of fiber. 

Since the restriction of $\omega$ on each fiber is closed under the fiberwise exterior differential $d_{Y}$,  and the exterior differential in the $
\Delta^{\star}\subset \mathbb{C}-$component  can be decomposed as the following 
$$d_{\mathbb{C}}=du\wedge \frac{\partial}{\partial u}+d\overline{u}\wedge \frac{\partial}{\partial \overline{u}}\  (\textrm{and}\ d=d_{\mathbb{C}}+d_{Y}), $$
we calculate
 \begin{eqnarray}\label{equ dw=0}0&=&d\omega
 \\&=& idu d\overline{u}[d_{Y}a+i\frac{\partial F^{0,1}}{\partial \overline{u}}-i\frac{\partial F^{1,0}}{\partial u}]+du[\frac{\partial \omega_{Y}}{\partial u}-d_{Y}F^{0,1}] \nonumber
+d\overline{u}[\frac{\partial \omega_{Y}}{\partial \overline{u}}-d_{Y}F^{1,0}].
  \end{eqnarray}
Then 
\begin{equation}
\frac{\partial \omega_{Y}}{\partial u}=d_{Y}F^{0,1},\ \ \frac{\partial \omega_{Y}}{\partial \overline{u}}=d_{Y}F^{1,0}.
\end{equation}
Take cohomology class $[ \ \cdot \ ]$ in $H^{1,1}[Y,\mathbb{C}]$, we find 
\begin{equation}
\frac{\partial [\omega_{Y}]}{\partial u}=\frac{\partial [\omega_{Y}]}{\partial \overline{u}}=0.
\end{equation}
This implies $[\omega_{Y}(u)]\in H^{1,1}[Y,\mathbb{C}]$ is constant in $u$.\\

Step 2: There is an unique smooth function $\phi$ on $\Delta^{\star}\times Y$
with fiber-wise $0-$average such that 
 \begin{equation}\label{equ fiber-wise iddbar}\omega_{Y}(u)=\omega_{Y,0}+i\partial_{Y} \overline{\partial}_{Y}\phi.\end{equation} This is because  $\omega_{Y}(u)$ and $\omega_{Y,0}$ are cohomologous.  Fiber-wise $i\partial_{Y} \overline{\partial}_{Y}-$solutions with $0-$average define a function on the fibration $\Delta^{\star}\times Y$ that  is  smooth by implicit function theorem. Please see  \cite[2.1]{TWY} about related discussion for  semi Ricci flat metrics introduced by \cite{GSVY}.  \\

Step 3: If $b^{1}(Y)=0$, the solution $\phi$ in \eqref{equ fiber-wise iddbar} actually solves the equation \eqref{equ form of w iddbar} on the fibration modulo pullback of a $(1,1)-$form on $\Delta^{\star}$ i.e. $i\partial \overline{\partial}$ of a pullback function  from $\Delta^{\star}$ (independent of $Y$). 

Namely, the decomposition \eqref{equ fiber-wise iddbar} and the vanishing of each component in \eqref{equ dw=0} says
 \begin{eqnarray}\label{equ 3 iddar}& &d_{Y}a+i\frac{\partial F^{0,1}}{\partial \overline{u}}-i\frac{\partial F^{1,0}}{\partial u}=0,
 \\& & i\partial_{Y} \overline{\partial}_{Y}\frac{\partial \phi}{\partial u} -d_{Y}F^{0,1}=0,\label{equ 1 iddar}
\\& & i\partial_{Y} \overline{\partial}_{Y}\frac{\partial \phi}{\partial \overline{u}}-d_{Y}F^{1,0}=0.\label{equ 0 iddar}
  \end{eqnarray}

The $(1,1)-$part of  \eqref{equ 0 iddar} must vanish,  which means the $(1,0)-$form $$i\partial_{Y} \frac{\partial \phi}{\partial \overline{u}}+F^{1,0}\ \ \textrm{is}\ \ \overline{\partial}_{Y}-\textrm{closed. Hence it is a holomorphic}\  (1,0)-\textrm{form  on}\  Y.$$ The assumption $b^{1}(Y)=0$  and Hodge decomposition implies vanishing of $H^{1}(Y, O_{Y})=H^{0}(Y, \Omega^{1}_{Y})$.  Therefore the above form must be $0$ i.e. 
\begin{equation} \label{equ 4 iddar}F^{1,0}=-i\partial_{Y} \frac{\partial \phi}{\partial \overline{u}}.\ \ \ \textrm{Similarly by}\ \eqref{equ 1 iddar},\ F^{0,1}=i\overline{\partial}_{Y} \frac{\partial \phi}{\partial u}. \end{equation}
  Plugging \eqref{equ 4 iddar} into \eqref{equ 3 iddar}, we find
  $$d_{Y}a- \partial_{Y} (\frac{\partial^{2} \phi}{\partial u \partial \overline{u}})- \overline{\partial}_{Y} (\frac{\partial^{2} \phi}{\partial u \partial \overline{u}})=0.$$
  This means $d_{Y}(a-\frac{\partial^{2} \phi}{\partial u \partial \overline{u}})=0$. Therefore
\begin{equation}\label{equ 5  iddar}a=\frac{\partial^{2} \phi}{\partial u \partial \overline{u}}+f(u)\end{equation}
where $f$ only depends on $u$,  not $Y$.  Plugging \eqref{equ fiber-wise iddbar}, \eqref{equ 4  iddar}, \eqref{equ 5  iddar} into the initial decomposition \eqref{equ -1 iddbar}, we find
 \begin{equation}\label{equ 6 iddar}
 \omega=f(u)idud\overline{u}+i\partial \overline{\partial}\phi+\omega_{Y,0}.
 \end{equation}
 
 Apply the ``un-conditional"  solvability Claim \ref{clm DeltaC solvability Hormander}, we find a function $h$ of only $u$ such that 
 $$i\partial \overline{\partial}h=f(u)idud\overline{u}.$$
Denote $\phi+h$ by $\phi_{0}$, the proof of  \eqref{equ form of w iddbar}
is complete. 

If $\omega$ is real, simply take the real part of $\phi_{0}$ so that \eqref{equ form of w iddbar} still holds.  In the $\sigma-$invariant case, because $\sigma$ is an bi-holomorphism (on both the fibration and the fiber), and the restriction $\omega_{Y,0}$ is still $\sigma-$invariant on $Y$, we   simply take  the average \begin{equation}\label{equ ave function sigma action}\underline{\phi}_{0}\triangleq \frac{\phi_{0}+\sigma^{\star}\phi_{0}+...+(\sigma^{ord_{\sigma}-1})^{\star}\phi_{0}}{ord_{\sigma}}\end{equation}
which is  $\sigma-$invariant  and still solves  \eqref{equ form of w iddbar}.   \end{proof}

\section{Constraint, iso-trivial ansatz, and proof of main results \label{sect ansatz}}
\subsection{Hein and Haskins-Hein-Nordstr\"om technique}
The  version  of H-HHN technique here  does not immediately work in the cylindrical case c.f.  \cite{HeinJAMES, HeinThesis, HHN}. It  requires the volume growth of the model metric to be strictly faster than linear (which we do not expect to be  sufficient). See how  $\beta>0$ in \eqref{equ lower bound on omega0} implies the modified real $(1,1)-$form $\omega+\eta$ is positive definite when error $<0$. 
\begin{lem}\label{lem constraint}Let  $(X,\omega)$ be a weak $\beta-$ALG K\"ahler manifold,  $R_{3}>0$ be large enough, and $\Delta^{\star}=\{|z|>R_{3}\}=\{|u|<\frac{1}{R_{3}}\}$. Suppose $\Upsilon$ is a (real) volume form such that $\frac{\Upsilon}{\omega^{d}}>0$ everywhere and $\frac{\Upsilon}{\omega^{d}}-1$ is integrable with respect to $\omega^{d}$. Then there is a  positive constant $t_{0}$ (see \eqref{equ t0} below) with the following property.  Consider the coordinate  $u=\frac{1}{z}$
near $\infty$ (denote  $u=\rho e^{i\theta}$) 
 and let  \begin{equation}\label{equ eta def}\eta\triangleq s\chi(\rho t)\rho d\rho\wedge d\theta=s\chi(\rho t) \frac{i}{2}du d\overline{u},\end{equation}
where $\chi(y)$ is a standard cutoff function that $\equiv 1$ when $2\leq y \leq 3$, but vanishes if $y\leq 1$ or $y\geq 4$. For any $t>t_{0}$, there is   $s$ such that
  $\omega_{0}\triangleq \omega+\eta$ is  K\"ahler  (positive definite)  that  satisfies constraint $\int_{X}(\omega^{d}_{0}-\Upsilon)=0$.
\end{lem}
Since $\eta$ is closed, $\omega+\eta$ is closed.  The parameter $t$ controls the distance from supp$\eta$ to the origin. However, $s$ and $\eta$ are  not necessarily positive. 
\begin{proof}Because $\eta\wedge \eta=0$, constraint \eqref{equ constraint} for $\omega_{0}\triangleq \omega+\eta$ is equivalent to the following integral condition:
\begin{equation}\label{equ eta}
d\int_{X}\eta\wedge \omega^{d-1}=-\int_{X}(\omega^{d}-\Upsilon).\ \ \ 
\end{equation}

Denote the right  side  of \eqref{equ eta} by the real number $Err$. 
   Let the  sign of $s$ be  that of $Err$. Then $\eta$ descends to the quotient $\Delta^{\star}/ \langle \sigma\rangle$ as it only depends on $\rho$ thus is invariant under the action of $\sigma$. Let $K\geq 10$ be a quasi-isometry constant such that 
\begin{equation}\label{equ quasi isomeric omega0 to the product}\frac{(\frac{i}{2}\cdot \frac{dud\overline{u}}{|u|^{2+2\beta}}+\omega_{Y})}{K}\leq \omega\leq K (\frac{i}{2}\cdot  \frac{dud\overline{u}}{|u|^{2+2\beta}}+\omega_{Y}).\end{equation}

If $Err=0$, we are done. The point is to find $s$ when $Err\neq 0$.\\

Case I: Suppose $Err>0$.  Integrating the left side of \eqref{equ eta} using the explicit  $\eta$ \eqref{equ eta def}  and  \eqref{equ quasi isomeric omega0 to the product}, we find 
\begin{eqnarray}& &\frac{2\pi\cdot d!s Vol(Y)}{ord_{\sigma} K^{d-1}}\int^{\infty}_{0}\chi(\rho t)\cdot \rho d\rho \nonumber
\\& \leq & d\int_{X}\eta\wedge \omega^{d-1}=Err\nonumber
\\& \leq & \frac{2\pi\cdot K^{d-1}(d!)s Vol(Y)}{ord_{\sigma}}\int^{\infty}_{0}\chi(\rho t)\cdot \rho d\rho \label{equ start of constraint case I}
\end{eqnarray}
 When $Err$ is negative, the reverse inequality holds. The $d!$ is the product of  the $d$ in the middle line and the denominator of the volume form $\frac{\omega^{d-1}_{Y}}{(d-1)!}$ on Y.   The order of the orbifold group appears because the integral on the orbifold is $\frac{1}{ord_{\sigma}}$ times the integral on the cover. 

Let $t>\frac{10}{R_{3}}$ be large enough such that  supp$\chi(\rho t)\subset\Delta^{\star}$. Using that $$\int^{\infty}_{0}\chi(\rho t)\cdot \rho d\rho =\frac{\Gamma_\chi}{t^{2}}$$
where $\Gamma_\chi=\int^{\infty}_{0} \chi(y)ydy$  only depends on the model cutoff function, the target integral $d\int_{X}\eta\wedge \omega^{d-1}_{0}$ is bounded from above and below respectively  by $$\frac{2\pi\cdot K^{d-1}(d!)s Vol(Y)\Gamma_\chi}{ord_{\sigma}\cdot t^{2}}\ \ \textrm{and}\ \ \frac{2\pi\cdot (d!)s Vol(Y)\Gamma_\chi}{K^{d-1}\cdot ord_{\sigma}t^{2}}.$$ 
If  $s>\frac{K^{d-1}ord_{\sigma}\cdot Err\cdot t^{2}}{2\pi\cdot (d!) Vol(Y)\Gamma_\chi}$, or $s<\frac{ord_{\sigma}\cdot Err\cdot t^{2}}{K^{d-1}\cdot 2\pi\cdot (d!) Vol(Y)\Gamma_\chi}$,  the target integral is either $> Err$ or $<Err$ respectively. Since it is a linear function in $s$, there exists a unique $s_{0}$ in the indicated range
\begin{equation}\label{equ end of constraint case I} ( \frac{ord_{\sigma}\cdot Err\cdot t^{2}}{K^{d-1}\cdot 2\pi\cdot (d!) Vol(Y)\Gamma_\chi},\ \frac{K^{d-1}ord_{\sigma}\cdot Err\cdot t^{2}}{2\pi\cdot (d!) Vol(Y)\Gamma_\chi}),
\end{equation}
 such that the target integral equals $Err$. 
$Err>0$ implies  positivity of $s$ and non-negativity of  $\eta$, therefore also the positivity of  $\omega+\eta$.
The point is when $Err<0$, we can still make it K\"ahler. \\

Case II. Suppose $Err<0$. With sign changes, similar argument as from \eqref{equ start of constraint case I} to \eqref{equ end of constraint case I} still produces  a $s_{0}$ in the interval 
\begin{equation}\label{equ range of s0} (-\frac{K^{d-1}ord_{\sigma}\cdot |Err|\cdot t^{2}}{2\pi\cdot (d!) Vol(Y)\Gamma_\chi},\ -\frac{ord_{\sigma}\cdot |Err|\cdot t^{2}}{K^{d-1}\cdot 2\pi\cdot (d!) Vol(Y)\Gamma_\chi}).
\end{equation}
Hence \eqref{equ eta} holds for $\eta$ with $s=s_{0}$.
On the support of $\chi(\rho t)$ that lies in the annulus $\{\rho\in (\frac{1}{t}, \frac{4}{t})\}$ (we assume $t>\frac{10}{R_{0}}$ in order for $supp \eta\subset \Delta^{\star}$), the metric $\omega$ satisfies
\begin{equation}\label{equ lower bound on omega0}
\omega\geq \frac{(\frac{i}{2} \frac{dud\overline{u}}{|u|^{2+2\beta}}+\omega_{Y})}{K}\geq  \frac{(\frac{i}{4}\frac{dud\overline{u}}{|u|^{2+2\beta}}+\frac{i}{4}dud\overline{u}\cdot \frac{1}{32^{2+2\beta}}t^{2\beta+2}+\omega_{Y})}{K}. 
\end{equation}
We  relaxed  $\frac{1}{|u|^{2\beta+2}}$ to its  lower bound $\frac{1}{32^{2+2\beta}}t^{2\beta+2}$ on $supp \chi$.  That $s_{0}=O(t^{2})$ implies $$\eta=O(t^{2})\cdot  \frac{i}{2}du\wedge d\overline{u}$$ as the cutoff function has magnitude $\leq 1$. But $\beta>0$ and the factor of $t^{2\beta+2}$ implies the growth of $\omega$ in $t$ is strictly faster than $O(t^{2})$.  Namely, using that $s_{0}$ lies in the interval \eqref{equ range of s0},  when 
 $$t^{2\beta}> \frac{1000\cdot 32^{2+2\beta}K^{d}ord_{\sigma}\cdot |Err|}{2\pi\cdot (d!) Vol(Y)\Gamma_\chi},$$
 we have $$\frac{i}{4}dud\overline{u}\cdot \frac{1}{32^{2+2\beta}}t^{2\beta+2}>\eta.$$
Therefore $\omega+\eta$ is still a K\"ahler form and
$$\omega\geq \frac{(\frac{i}{2}\cdot\frac{dud\overline{u}}{|u|^{2+2\beta}}+\omega_{Y})}{2K}$$
on $supp \eta$. 
In summary, no matter what the sign of  $Err$ is, let 
\begin{equation}\label{equ t0}t>t_{0}\triangleq \max\{\frac{10}{R_{0}},  [\frac{1000\cdot 32^{2+2\beta}K^{d}ord_{\sigma}\cdot |Err|}{2\pi\cdot (d!) Vol(Y)\Gamma_\chi}]^{\frac{1}{2\beta}}\},\end{equation}
$\omega+\eta$ is K\"ahler. \end{proof}

 \subsection{Hein decay estimate  for Tian-Yau solution} 

An  $ALG(\frac{2\pi}{n})-$metric has $(K,2,0)-$polynomial growth  \cite[Section 1]{TY1} and  positive  injective radius due to curvature bound and volume ratio lower bound for all balls with radius $< 1$ \cite{Grant, CGT} etc.  The bound on covariant derivatives of curvatures also implies the quasi-finite geometry of order $2+\frac{1}{2}$  \cite[Proposition 1.2]{TY1}. The general existence \cite[Theorem 1.1]{TY1} provides a smooth solution $\phi$ to the target  Monge-Ampere equation \eqref{equ Target MA} such that 
 $\omega_{\phi}$ is quasi isometric to the reference metric $\omega_{0}$ i.e. there is a positive constant $C$ such that 
 \begin{equation}\frac{\omega_{0}}{C}\leq \omega_{0}+i\partial \overline{\partial}\phi\triangleq \omega_{\phi}\leq C\omega_{0}.\end{equation}

The target  equation \eqref{equ Target MA} can be written as
 \begin{equation}\label{equ difference between volume forms}
 \frac{(\omega_{0}+i\partial \overline{\partial}\phi)^{d}-\omega_{0}^{d}}{\omega_{0}^{d}}=e^{f}-1.
 \end{equation}
 We further write it as a second order elliptic equation with co-efficients depending on the unknown, as  \cite{TY1, TY2, HHN, Kovalev} etc. 
 \begin{equation}\label{equ Kovalev form of MA}
 \frac{i\partial \overline{\partial}\phi\wedge[\omega^{d-1}_{\phi}+...+\omega_{0}^{j}\wedge \omega_{\phi}^{d-1-j}+...+\omega_{0}^{d-1}]}{\omega_{0}^{d}}=e^{f}-1.
 \end{equation}
 See  related techniques in   \cite[Theorem 1.1]{TY1},  \cite{HeinThesis},  \cite{Kovalev},    and \cite{HHN}. 
  
 In the energy decay \eqref{equ Dirichlet energy decay}, abusing notation, we still denote  $\min\{\beta_{0},\frac{1}{1000}\}$ by $\beta_{0}$. Under the poly-cylinder that utilizes the Tian-Yau atlas on the fiber direction $Y$, the   $i\partial \overline{\partial}\phi$ below  is viewed as a matrix under the chart. This norm is equivalent to the one  viewing it as a tensor with derivative being co-variant derivatives, by the bounds on the pullback metric under Tian-Yau coordinates. See Appendix \ref{Appendix Tian Yau atlas} for a more comprehensive review and the standard norms. 
 \begin{lem}(Hein \cite{HeinThesis}) \label{equ initial decay}
$ |i\partial \overline{\partial}\phi|_{C^{k}[A(r,\infty)]}\leq \frac{C_{\beta_{0},k}}{r^{\frac{\beta_{0}}{2}}}$.

 \end{lem}

\begin{proof} This is direct from  Hein Dirichlet energy decay \eqref{equ Dirichlet energy decay}.  

Neumann Poincare inequality on $B_{euc,x}(\mu_{2})$ implies
\begin{equation}\label{equ energy decay 0}
\int_{B_{euc,x}(\frac{\mu_{2}}{2})}|\phi-c_{\phi}|^{2}\leq C\int_{B_{euc,x}(\mu_{2})}|\nabla \phi|^{2}\leq C_{\beta_{0}}r^{-\beta_{0}},
\end{equation}
where $c_{\phi}$ is a constant depending on $\phi$. Tian-Yau higher-order bound \cite[Theorem 1.1, last sentence]{TY1}
implies  the co-efficients of \eqref{equ Kovalev form of MA}, as a second order elliptic equation in $\phi$,  are bounded in $C^{1}[B_{euc,x}(\frac{\mu_{2}}{4})]$ in terms of the Ricci potential/volume density $f$, and the ansatz $\omega_{0}$. Then interior Schauder estimate for linear equation and asymptotics of $e^{f}-1$ says
\begin{equation}|\phi-c_{\phi}|_{C^{2,\frac{1}{2}}[B_{euc,x}(\frac{\mu_{2}}{8})]}\nonumber \leq C|e^{f}-1|_{C^{\frac{1}{2}}[B_{euc,x}(\frac{\mu_{2}}{4})]}+C|\phi-c_{\phi}|_{L^{2}[B_{euc,x}(\frac{\mu_{2}}{4})]}\nonumber
\leq  \frac{C_{\beta_{0}}}{r_{x}^{\frac{\beta_{0}}{2}}}.
\end{equation}
 Lemma \ref{equ initial decay} is proved when $k=2$. By the Schauder bounds on the metric under balls in 
Tian-Yau coordinate (Section \ref{Appendix Tian Yau atlas}),   
differentiating the equation, we find higher order bounds. 
\begin{equation}|\phi-c_{\phi}|_{C^{k+2, \frac{1}{2}}[B_{euc,x}(\frac{\mu_{2}}{10})]}\leq  \frac{C_{\beta_{0},k}}{r_{x}^{\frac{\beta_{0}}{2}}}
\end{equation}
This particularly implies $ |i\partial \overline{\partial}\phi|_{C^{k}[A(r,\infty)]}\leq \frac{C_{\beta_{0},k}}{r^{\frac{\beta_{0}}{2}}}$. \end{proof}

\subsubsection{Finishing the proof of main results}
\begin{proof}[Proof of  Theorem \ref{Thm},  Corollary \ref{Cor >2}, and existence part of Corollary \ref{Cor 2}:].\\ \\
Step 1: ALG metric that is equal to a (Ricci flat) ALG model outside a compact set. 

The assumed K\"ahler metric yields an iso-trivial ALG metric $\omega_{isotr}$. 
Namely, the difference of the expression \eqref{equ form of w iddbar} for $\omega$
and the ALG model  \eqref{equ model orb product metric} yields 
 \begin{eqnarray}& &\omega-\frac{i}{2} dzd\overline{z}-\omega_{RF, Y}=i\partial \overline{\partial}\phi_{0}-\frac{i}{2} dzd\overline{z}+\omega_{Y}-\omega_{RF,Y}\nonumber
 \\&=&i\partial \overline{\partial}(\phi_{0}-\phi_{RF,Y}-\frac{|z|^{2}}{2}). 
 \end{eqnarray}
 where $\omega_{RF,Y}$ is the unique Ricci flat K\"ahler metric cohomologous to  $\omega_{Y}$, and  $\phi_{RF,Y}$ is the Ricci flat potential: 
 \begin{equation}
 \omega_{RF,Y}-\omega_{Y}=i\partial \overline{\partial}\phi_{RF,Y}, \textrm{where}\ \omega_{RF,Y}^{d-1}=\nu_{d-1}\Omega_{Y}\wedge \bar{\Omega}_{Y}\ \textrm{and}\ \int_{Y}\omega_{RF,Y}=0.
 \end{equation}

To obtain $\phi_{0}$, we applied Lemma \ref{lem strong iddbar} with the coordinate change $\rho=\frac{1}{r},\ u=\frac{1}{z}$ near $z=\infty$. Gluing Lemma \ref{lem Gluing for ansatz} says that when $t$ is large enough, 
$$\omega_{isotr}\triangleq \omega+i\partial \overline{\partial}\chi (-\phi_{0}+\phi_{RF,Y}+\frac{|z|^{2}}{2})+tb_{0}.$$
is the desired iso-trivial K\"ahler metric, where $\chi$ is the cutoff function supported on the end and vanishes in a compact set containing the exceptional divisor.\\\

Step 2: Ansatz with constraint.

The desired ansatz $\omega_{0}$ is obtained applying Lemma \ref{lem constraint} to $\omega_{isotr}$ i.e. 
\begin{equation}\label{equ RFALG potential 0}
\omega_{0}=\omega_{isotr}+\eta= \omega+i\partial \overline{\partial}\chi (\phi_{0}-\phi_{RF,Y}+\frac{|z|^{2}}{2})+tb_{0}+\eta.
\end{equation}

Since the difference between $\omega_{0}$ and $\omega_{isotr}$ is compactly supported, $\omega_{0}$ is iso-trivial as well. Since $f$ is compactly supported,
it meets the decay condition required in Tian-Yau solvability \cite[Thm 1.1]{TY1}.\\  

Step 3: Tian-Yau solution, Hein decay, and Schwartz decay.

 Apply Tian-Yau solvability to the isotrivial ansatz $\omega_{0}$ with constraint, we find the solution. We can write down the difference between the ALG Ricci flat metric $\omega_{RFALG}$ with the initial K\"ahler metric $\omega$ more explicitly. Applying the radial solvability in Claim \ref{clm DeltaC solvability Hormander} to the two radial compact supported bump forms in \eqref{equ RFALG potential}, we find smooth radial functions $\phi_{rb,g} $ and $\phi_{rb,c}$ (potentials of radial bump forms for gluing and constraint respectively) such that 
$$tb_{0}=i\partial \overline{\partial}\phi_{rb,g} ,\ \ \eta=i\partial \overline{\partial}\phi_{rb,c}.$$
Then 
\begin{equation}\label{equ RFALG potential regarding w -1}
\omega_{RFALG}-\omega=i\partial \overline{\partial}[\chi (\phi_{0}-\phi_{RF,Y}+\frac{|z|^{2}}{2})+\phi_{rb,g}+\phi_{rb,c}+\phi_{TY}],
\end{equation}
and the ALG Ricci flat potential $\phi_{RFALG}$ \eqref{equ RFALG potential} equals 
\begin{equation}\label{equ RFALG potential regarding w} \chi (\phi_{0}-\phi_{RF,Y}+\frac{|z|^{2}}{2})+\phi_{rb,g}+\phi_{rb,c}+\phi_{TY}.\end{equation}
Except the Tian-Yau potential $\phi_{TY}$ \cite[Theorem 1.1]{TY1} relative to  $\omega_{0}$, the other potential functions \eqref{equ RFALG potential regarding w} relative to the initial K\"ahler metric $\omega$ in Theorem \ref{Thm} are supported on the end i.e. they vanish in a compact set containing the exceptional divisor.  Particularly, they  are smooth functions on the crepant resolution  $X$. Since the resolution map is not necessarily a smooth orbifold map, the pullback of a smooth function from the orbifold to the resolution is not known to still be smooth if it does not vanish near the fixed  locus. 

Hein decay (Lemma \ref{equ initial decay}) fulfills the condition in  Schwartz decay Proposition  \ref{prop Schwartz decay for grav instantons}. Therefore the metric has Schwartz decay toward the ALG model. Proof for existence is complete.\\

Step 4: Monge-Ampere partial uniqueness.

By metric equivalence between the ansatz $\omega_{0}$ and $\omega_{RFALG}$, the Poincare inequality Lemma \ref{lem Poincare} for ALG model yields the weak Poincar\'e inequality for $\omega_{RFALG}$. This  fulfills the requirement for  uniqueness of bounded $C^{1,1}-$solutions \cite{YW} that implies  the desired uniqueness here.\\

Step 5: Examples of $ALG_{\infty}$ Ricci flat K\"ahler $3-$folds via $K3$ with purely non-symplectic automorphism of finite order. 

 Because all $K3'$s with a non-symplectic automorphism $\sigma$ of finite order is polarized i.e. admits positive line bundle $L_{K3}$ (for example, see \cite{AST}).  The ``average" 
$$L_{0}\triangleq L_{K3}\otimes \sigma^{\star}L_{K3}\otimes...\otimes (\sigma^{ord_{\sigma}-1})^{\star}L_{K3}$$ is a $\sigma-$invariant positive line bundle  on the $K3$. Denote the $\sigma-$invariant positive curvature form of $L_{0}$ by $\omega_{K3,\sigma}$. The 
ALG model $$\frac{i}{2}\lambda^{2} dzd\overline{z}+\omega_{K3,\sigma}$$ is a K\"ahler orbifold metric on the $3-$orbifold $$\frac{\mathbb{C}\times K3}{\langle\sigma\rangle}$$
whose orbifold groups are obviously in $SL(3,\mathbb{C})$  by our construction.  
Then  discussions  \cite[Page 137 above 6.6.1]{Joyce} say there is a K\"ahler crepant resolution. The point (of Lemma \ref{lem strong iddbar}) is that  we only need a K\"ahler metric at this step, and  no completeness  is required.  For the examples in Corollary \ref{Cor 2}, the K\"ahler metric on $X$ can be taken as the Bergman metric \eqref{equ Bergman} below.

By \cite[Main Theorem 3, Corollary 5, Table 1]{MO}, an integer among the list  \eqref{equ composite angles} is the order of an purely non-symplectic automorphism on a $K3$.  Corollary \ref{Cor >2} and existence part of Corollary \ref{Cor 2} are proved.  \end{proof}



\section{Bergman  metric \label{sect orbifold metric}}

\subsection{ $A_{1}-$singularity and the crepant resolution\label{subsect Basics of $A_{n-1}-$singularity and the crepant resolution}}

 The action of $-Id$ on $\mathbb{C}^{2}$ defines the $A_{1}-$singularity $\frac{\mathbb{C}^{2}}{\langle \pm Id \rangle}$. This orbifold is isomorphic to  the affine variety $$V=\{(u_{0},u_{1},u_{2})|u_{0}u_{2}=u^{2}_{1}\subset \mathbb{C}^{3}\}$$ via the map 
 \begin{equation}\label{equ embed standard A1 cone}u_{0}=x^{2},\ u_{1}=xy,\ u_{2}=y^{2}.
 \end{equation}
We define two copies of $\mathbb{C}^{2}$ and  denote them by  $\mathbb{C}_{0}^{2}$ and $\mathbb{C}_{1}^{2}$, with coordinates $(\lambda_{0}, \mu_{2})$ and $(\lambda_{1}, \mu_{1})$ respectively. The  crepant resolution of  $A_{1}-$singularity is the gluing of them  by the map
\begin{equation}\label{equ gluing map C2}
\lambda_{1}=\frac{1}{\mu_{0}},\ \mu_{1}=\lambda_{0}\mu^{2}_{0}. 
\end{equation}

\begin{itemize}
\item The map
\begin{equation*}
u_{0}=\lambda_{0},\ u_{1}=\lambda_{0}\mu_{0},\ u_{2}=\lambda_{0}\mu^{2}_{0}. 
\end{equation*}
 is an isomorphism $\mathbb{C}_{0}^{2}\setminus \{\lambda_{0}= 0\}\longrightarrow V\setminus \{x= 0\}$. It maps the whole $\mu_{0}-$axis to the origin. The inverse  when $x\neq 0$ is 
 \begin{equation}\label{equ morphism to the affine variety 0} \lambda_{0}=x^{2},\ \ \mu_{0}=\frac{y}{x}. 
 \end{equation}
 
\item  The map 
\begin{equation*}
u_{0}=\lambda^{2}_{1}\mu_{1},\ u_{1}=\lambda_{1}\mu_{1},\ u_{2}=\mu_{1}. 
\end{equation*}
 is an isomorphism $\mathbb{C}_{0}^{2}\setminus \{\mu_{1}= 0\}\longrightarrow V\setminus \{y= 0\}$.  It maps the whole $\lambda_{1}-$axis to the origin. 
 The inverse when $y\neq 0$ is 
 \begin{equation} \label{equ morphism to the affine variety 1}
 \lambda_{1}=\frac{x}{y},\ \ \mu_{1}=y^{2}.
 \end{equation}

\end{itemize}

The $\mu_{0}-$axis in $\mathbb{C}_{0}^{2}$ 
is glued to the $\lambda_{1} -$axis in  $\mathbb{C}_{1}^{2}$ along the non-zero part via the transition map
$\lambda_{1}=\frac{1}{\mu_{0}}$, in conjunction with   \eqref{equ gluing map C2}. This is  the exceptional (smooth) $\mathbb{P}^{1}$ which is  a  $(-2)-$curve. \begin{figure}[h]\begin{center}  \begin{tikzpicture}
\node at (-4.6,1) {$\lambda_{0}$};
\node at (-4,-0.4) {$\mu_{0}$};
  \draw[->,semithick] (-6,0) -- (-4,0);
    \draw[->,semithick] (-5,-1) -- (-5,1);
    \node at (1.4,1) {$\lambda_{1}$};
\node at (2,-0.4) {$\mu_{1}$};
  \draw[->,semithick] (0,0) -- (2,0);
    \draw[->,semithick] (1,-1) -- (1,1);
    \draw [bend left, dashed,-]
(-4.8,0) -- (1,0.5);
\end{tikzpicture}  
\end{center}
\caption{Manifold structure for crepant resolution of $A_{1}-$singularity}\label{picture A1}
\end{figure}
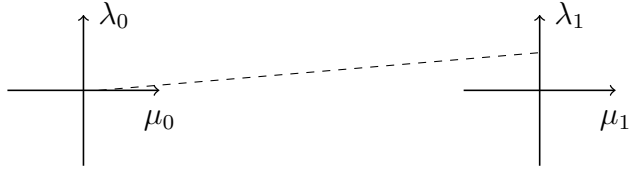
The isomorphism from the resolution $\widetilde{\frac{\mathbb{C}^{2}}{\langle\sigma \rangle}}=\mathbb{C}_{0}^{2}\underline{\cup}\mathbb{C}_{1}^{2}$ to the total space $O(-2)\rightarrow \mathbb{P}^{1}$ is : 
$$ \left\{\begin{array}{c}(\lambda_{0}, \mu_{0})\rightarrow \lambda_{0}(1,\mu_{0})^{\otimes 2}\\
(\lambda_{1}, \mu_{1})\rightarrow \mu_{1}(\lambda_{1}, 1)^{\otimes 2}.
\end{array}\right.$$
 We define the $0$ in $\lambda_{1}-$axis in $\mathbb{C}_{1}^{2}$ as the ``$0$" in the exceptional divisor, and the $0$ in $\mu_{0}-$axis as the ``$\infty$".  See \cite{Reid} for more explicit information on Du Val singularities. 

\subsection{Nikulin poly-cylinder}
At a fixed point $p$ of the Nikulin involution on a $K3$, we call a coordinate neighborhood  $\{U_{p}, (y, v)\}$ a \textit{Nikulin coordinate} at $p$, if $p$ maps to $(0,0)$ in the chart, and   the Nikulin-involution $\sigma$ is diagonal therein: 
\begin{equation}\sigma\cdot (y,v)= (y,v)\cdot \left| \begin{array}{cc}-1& \\
& 1 \\
\end{array}\right|.\end{equation}
 Such a coordinate always exists \cite{Nikulin}.  On the other hand, the fixed points of $\sigma$ form a disjoint union of  holomorphic curves. Let $\mathcal{C}_{i}$ denote a component of the fixed locus $\mathcal{C}$ of $\sigma$ i.e. a (connected) curve  point-wisely fixed by $\sigma$. In a  Nikulin coordinate at a $p\in \mathcal{C}$, the curve  is defined by $y=0$. Consider a poly-cylinder in $y$ and $v$ of equal radius (contained in the coordinate chart) with closure in a Nikulin coordinate neighborhood:
   $$PC_{p}(r)\triangleq D_{y}(r)\times D_{v}(r),$$
   where $D_{y}(r)$ is open disk of radius $r$ in the complex plane with coordinate $y$ centered at $y=0$, and $D_{v}(r)$ is likewise.  We call it a \textit{Nikulin poly-cylinder}. The transition function between two Nikulin poly-cylinders at a point $p$ has a Taylor expansion near the origin (that corresponds to $p$):
\begin{equation}
\widehat{y}=a_{1}(v)y+a_{3}(v)y^{3}+a_{5}(v)y^{5}...
\end{equation}
 The reason why only odd degree terms in $y$ appear is the invariance
$$ \widehat{y}(-y,v)= -\widehat{y}(y,v).$$
 Non-vanishing of Jacobian of the transition function says $$a_{1}(v)=\lim_{y\rightarrow 0}\frac{\widehat{y}}{y}(y,v)=\frac{\partial \widehat{y}}{\partial y}(0,v)\neq 0$$ 
for any $v$ on the overlap on  $\mathcal{C}_{i}$. It is the transition function for the the co-normal bundle  $N^{\star}_{\mathcal{C}_{i}}$ of $\mathcal{C}_{i}$. 
Let $Exc$ denote the dis-joint union $\cup_{i}Exc_{i}$,  where  $Exc_{i}\triangleq \pi^{-1}\mathcal{C}_{i}$.  
\subsection{Resolution along fixed curves and Nikulin-Reid poly-cylinders}
  We simply do   product resolution which is independent of coordinate. $x$ still parametrizes $\mathbb{C}$. Abusing terms, we also call 
\begin{equation}\label{equ triple Polycyl equal radius} D_{x}(r)\times D_{y}(r) \times D_{v}(r)\end{equation}
and the quotient $$\widetilde{\frac{D_{x}(r)\times D_{y}(r)}{\langle \sigma \rangle}}\times D_{v}(r)\triangleq PC_{resol}$$
 Nikulin  poly-cylinders.  According to the inverse resolution morphisms \eqref{equ morphism to the affine variety 0}, \eqref{equ morphism to the affine variety 1} and the transition,  
 corresponding to the manifold structure of crepant resolution of $\frac{\mathbb{C}^{2}}{\{\pm Id\}}$, we have the decomposition
 \begin{equation}\nonumber PC_{resol}
 = PC_{resol,0}\cup PC_{resol,1}.\label{equ Nikulin Reid PC} \end{equation} 
where  $PC_{resol,i}$ are parametrized by coordinates $(\lambda_{i},\mu_{i})$ of $\mathbb{C}_{0}^{2}$ and $\mathbb{C}_{1}^{2}$.  They are called  \textit{Nikulin-Reid poly-cylinders}. Unlike \eqref{equ triple Polycyl equal radius},  we do not know whether they are of equal radius. 
 
  Let $(x,y,v)$ and $(x,\widehat{y},\widehat{v})$
 be two Nikulin  poly-cylinders. On the overlap, the transition map lifts to the  resolution because $\frac{\widehat{y}}{y}$ extends holomorphically in $y^{2}=\lambda_{0}\mu^{2}_{0}=\mu_{1}$ and $v$.  Moreover,  $$\widehat{\lambda}_{0}=\lambda_{0}=x^{2},\ \ \ \ \widehat{\mu}_{0}=\mu_{0}\cdot \frac{\widehat{y}}{y}, $$
$$\widehat{\lambda}_{1}=\lambda_{1}(\frac{\widehat{y}}{y})^{-1},\ \ \ \ \widehat{\mu}_{1}=\mu_{1}\cdot (\frac{\widehat{y}}{y})^{2}. $$

When $y^{2}=0$, the transition  between Nikulin-Reid poly-cylinders becomes the fiberwise linear map:
 $$\widehat{\lambda}_{0}=\lambda_{0},\ \ \ \ \widehat{\mu}_{0}=\mu_{0}\cdot a_{1},\ \ \ \widehat{\lambda}_{1}=\frac{\lambda_{1}}{a_{1}(v)},\ \ \ \ \widehat{\mu}_{1}=\mu_{1}\cdot a^{2}_{1}.$$
 Therefore, the map $$\lambda_{1}\longrightarrow [1\oplus \lambda_{1}dy],\ \ \ \mu_{0}\longrightarrow [\mu_{0}\oplus dy]$$
is  independent of coordinate, and yields an isomorphism 
$$Exc_{i}\longrightarrow \mathbb{P}[O_{\mathcal{C}_{i}}\oplus N^{\star}_{\mathcal{C}_{i}}]=\mathbb{P}[N_{\mathcal{C}_{i}}\oplus O_{\mathcal{C}_{i}}]=\mathbb{P}[\textrm{Normal Bundle of}\ \mathcal{C}_{i}].$$
Adjunction says   $$N_{\mathcal{C}_{i}}=[\mathcal{C}_{i}]|_{\mathcal{C}_{i}}=K_{\mathcal{C}_{i}},\  \ \ N^{\star}_{\mathcal{C}_{i}}=K^{\star}_{\mathcal{C}_{i}}=T^{1,0}_{\mathcal{C}_{i}}\ (\textrm{holomorphic tangent bundle of}\ \mathcal{C}_{i}).$$

\subsection{Projectivity \label{sect Projectivity}}

  Let $Bl^{O\times Fix(\sigma)}(\mathbb{P}^{1}\times K3)$ denote the blow up of $\mathbb{P}^{1}\times K3$ along the curve  $O\times Fix(\sigma)$. 
   It is still a smooth projective variety. See  \cite[Remark 1, 6.2.2]{Shafarevich}, \cite[II. Proposition 7.16]{Hartshorne} for related discussions.  Then there is a positive line bundle $L_{0}\rightarrow Bl^{O\times Fix(\sigma)}(\mathbb{P}^{1}\times K3).$ By the very ampleness discussions in \cite[II Propsition 7.10 (b), Remark 5.16.1]{Hartshorne},  and that the invertible sheaf ``$O_{P}(1)$" in \cite[II Propsition 7.10 (b)]{Hartshorne} is of the divisor $-Exc$ (see \cite[4]{Vakil}), 
  we can simply take 
  \begin{equation}\label{equ ample line bundle on the compactification}L=[-Exc]\otimes m\pi^{\star}O_{\mathbb{P}^{1}}(1)\otimes m\pi^{\star}L_{K3}\end{equation} for sufficiently large positive $m$ regarding a $\sigma-$invariant  ample/positive line bundle $L_{K3}$ on the $K3$.   The quotient 
  \begin{equation}
  \frac{Bl^{O\times Fix(\sigma)}(\mathbb{P}^{1}\times K3)}{\langle \pm Id \rangle}\triangleq \overline{X}
  \end{equation}
is a complex orbifold, and  is also  a  projective variety  by Kodaira-Baily embedding theorem \cite{Baily}.  
By  coordinates for blow ups of a smooth complex manifold  \cite[4.6]{GH},  and of the orbifold with $A_{1}-$singularities in Section \ref{subsect Basics of $A_{n-1}-$singularity and the crepant resolution},  it is routine to check
  \begin{equation}\label{equ blow up and my ALG mfld}
  \frac{Bl^{O\times Fix(\sigma)}(\mathbb{C}\times K3)}{\langle \pm Id \rangle}=\widetilde{\frac{\mathbb{C} \times K3}{\langle \pm Id \rangle}}=X. 
  \end{equation} i.e. the crepant resolution of the $Z_{2}-$quotient is isomorphic to the $Z_{2}-$quotient of the blow up. Namely, similar to \eqref{equ morphism to the affine variety 0} its subsequent material, let $x,\ y,\ v$ be Nikulin coordinate near a point on $O\times Fix(\sigma)$, we have 
the following coordinates $z_{i},w_{i}$ near the exceptional divisor of the blow up of $\mathbb{P}^{1}\times K3$:
$$z_{0}=x,\ w_{0}=\frac{y}{x}\ \textrm{when}\ x\neq 0,\ \textrm{and} \ z_{1}=\frac{x}{y},\ w_{1}=y,\ \textrm{when}\ y\neq 0.$$
Moreover, we have the transition on the overlap of Nikulin- coordinates:
\begin{equation}\label{equ coordinate usual blow up 0} \widehat{z}_{0}=z_{0}=x,\ \ \ \ \widehat{w}_{0}=w_{0}\cdot \frac{\widehat{y}}{y}\ ; \ \widehat{z}_{1}=z_{1}(\frac{\widehat{y}}{y})^{-1},\ \  \widehat{w}_{1}=w_{1}\cdot \frac{\widehat{y}}{y}\ . \end{equation}
   The action of $\sigma$ on $\mathbb{P}^{1}\times K3$, locally near a point on $O\times Fix(\sigma)$ given by 
   $$\sigma (x,y,v)=(-x,-y,v),$$
obviously lift to the blow-up i.e.   
\begin{equation}\label{equ lift} \sigma (z_{0},w_{0},v)=(-z_{0},w_{0},v),\ \sigma (z_{1},w_{1},v)=(z_{1},-w_{1},v),  \end{equation}
and is independent of the coordinate chosen, as when we replace $x,\ y,\ \widehat{x},\ \widehat{y}$ by $-1$ of themselves in \eqref{equ coordinate usual blow up 0} and  \eqref{equ lift}, the negative signs cancel out. The morphism 
\begin{equation}\label{equ covering orbifold} Bl^{O\times Fix(\sigma)}(\mathbb{P}^{1}\times K3):\longrightarrow \widetilde{\frac{\mathbb{P}^{1} \times K3}{\langle \pm Id \rangle}},
\end{equation}
locally near the exceptional divisor, is defined by 
$$\lambda_{0}=z^{2}_{0},\ \mu_{0}=w_{0},\ \lambda_{1}=z_{1},\ \mu_{1}=w^{2}_{1}.$$
\begin{rmk}\label{rmk conic Borbon Spotti} Fu-Namikawa \cite{FN}  showed   existence of lifting for   surface singularities.  By $d\lambda_{0}=2z_{0}dz_{0}$, we find  $d\lambda_{0}d\overline{\lambda}_{0}=4|z_{0}|^{2}dz_{0} d\overline{z}_{0}$. Since $\lambda_{0}$ is the defining holomorphic/regular function of the exceptional divisor in Nikulin coordinate,  the pullback of the ALG Ricci flat metric from $X$ to  $Bl^{O\times Fix(\sigma)}(\mathbb{C}\times K3)$ is still a Ricci flat K\"ahler metric with an ALG end and conic singularity along the exceptional divisor of angle $4\pi$. 
\end{rmk}

The Bergman metric   is written as 
\begin{equation}\label{equ Bergman}\omega_{Bergman}=i\partial\overline{\partial}\log (|s_{0}|^{2}+...+|s_{n}|^{2}),\end{equation}
where $s_{i}\in H^{0}[\overline{X}, L]$. 
 The  biregular map $\Phi$ produced by Baily \cite[Global imbedding theorem]{Baily} locally generates the whole stalk of invariant holomorphic functions \cite[below local imbedding theorem]{Baily}. Restricted to $X$, it is a genuine complex manifold embedding.

\section{Topology and table of  examples \label{sect Topology}}
\subsection{Fundamental group}

The global quotient $\frac{\mathbb{C}\times Y}{\langle \sigma \rangle}$, denoted by $Orb$,  is normal. The result of   Kollar \cite[Theorem 7.8]{Kollar} says the induced homorphism $$p_{\star}:\ \pi_{1}[X]\rightarrow \pi_{1}[Orb]$$
between the fundamental groups is an isomorphism. The radial retraction of the $\mathbb{C}-$component to the origin is invariant under the group action, Thus it induces an isomorphism 

$$\pi_{1}[X]\rightarrow \pi_{1}[Y/\langle \sigma \rangle].$$

\begin{lem}\label{lem surj fundamental group}Let $Y$ be a smooth compact path connected  manifold. Let $\sigma$ be  diffeomorphism  of $Y$ with prime order. Suppose $\sigma$ has at least one fixed point $x_{0}$. Then the induced homorphism $\pi_{1}(Y)\longrightarrow  \pi_{1}(\frac{Y}{\langle \sigma\rangle})$ is surjective. 
\end{lem}
Because $Y$ is simply connected and compact,  Lemma \ref{lem surj fundamental group} implies both the orbifold  and   the resolution $X$ are simply-connected. 
\begin{proof}[Proof of Lemma \ref{lem surj fundamental group}:] The prime order condition implies fixed locus of a non-zero power $\sigma^{k}$ is the same as fixed locus of $\sigma$. This is straight-forward: there are integers $a,b$ such that 
$ak-bp=1$, therefore

$$\sigma^{k}(x)=x\ \textrm{if and only if}\ \sigma(x)=\sigma^{ak-bp}(x)=x.$$

Thus 
\begin{equation}\label{equ fix pt sigma group}Fix(\sigma)=Fix(\langle \sigma \rangle)\triangleq \cup_{1\leq k \leq ord_{\sigma}-1}Fix (\sigma^{k})= Fix(\sigma^{j})\ \textrm{for any}\ 1\leq j\leq ord_{\sigma}-1.
\end{equation}Henceforth we denote it by the first. 

The observation is that a loop $$\gamma:\ [0,1]\longrightarrow \frac{Y}{\langle \sigma \rangle}\ \textrm{based at the image}\   [x_{0}]\in \frac{Y}{\langle \sigma \rangle} $$ of the fixed point  can be lifted to a loop in $Y$ based at $x_{0}$. Namely, $x_{0}$ is the only pre-image of $[x_{0}]\in \frac{Y}{\langle \sigma \rangle} $ under the quotient $\rho:\ Y\longrightarrow  \frac{Y}{\langle \sigma \rangle} $. As  $\langle \sigma \rangle$ is cyclic and finite, the restricted map away from the fixed locus: $Y-Fix\sigma \longrightarrow  \frac{Y-Fix(\sigma)}{\langle \sigma \rangle} $ is a covering. There is a closed subset $S\subset [0,1]$ such that 
$\gamma |_{S}$ lies on $\rho[Fix(\sigma)]$,
and  $\gamma |_{[0,1]\setminus S}$ lies on $\frac{Y-Fix( \sigma )}{\langle \sigma \rangle}$.  Any open (closed) subset in $[0,1]\subset \mathbb{R}$ is a countable union of bounded open (closed) intervals in $[0,1]$. Since $\gamma(0)=\gamma(1)=x_{0}\in S$, we   denote 
$$[0,1]\setminus S=\cup_{k=1}^{\infty}(a_{k},b_{k}).$$
By \eqref{equ fix pt sigma group}, $\gamma |_{S}$ automatically lifts by the identity homeomorphism
$$\pi: Fix(\sigma)\longrightarrow Fix(\langle \sigma \rangle)$$
to $Y$.  For each $k$, denote a lift (by property of covering space) of $$\gamma:\ (a_{k},b_{k})\longrightarrow \frac{Y}{\langle \sigma \rangle}$$ 
by $\gamma_{Y}(a_{k},b_{k})$. The point is to show it ``matches" the trivial lift 
$\gamma |_{S}$ at the end points.  Equip $Y$ with a $\sigma-$invariant Riemannian metric, by the average $\frac{1}{p}\Sigma_{i=0}^{p-1}(\sigma^{i})^{\star} g_{0}$, where $g_{0}$ is a smooth metric on $Y$. Because $Y$ is compact and $\sigma$ is continuous, for any sequence  $t_{j}\rightarrow a^{+}_{k}$,  $\gamma_{Y}(t_{j})$ sub-converges to point $x_{k,left}$
whose image in $\frac{Y}{\langle \sigma \rangle}$ is $\gamma_{a_{k}}$. Thus $x_{k,left}=\gamma_{a_{k}}$ as fixed point in $Y$.  Similarly, for any sequence  $t_{j}\rightarrow b^{-}_{k}$,  $\gamma_{Y}(t_{j})$ sub-converges to point $x_{k,right}$
whose image in $\frac{Y}{\langle \sigma \rangle}$ is $ \gamma_{b_{k}}$. Thus 
$x_{k,right}= \gamma_{b_{k}}$ as fixed point in $Y$. Hence the limits exist and is correct. This means   the lift $\gamma_{Y}(t)$ is continuous at the end points.  \end{proof}

\subsection{Betti numbers}
\begin{prop}\label{prop Betti numbers} As in Corollary \ref{Cor 2},  let $(K3,\sigma)$ be a $K3-$surface with a non-symplectic Nikulin involution $\sigma$. 
Then the crepant resolution $X\triangleq \widetilde{\frac{\mathbb{C}\times K3}{\langle \sigma \rangle}}$ (which is simply-connected) has Betti number
\begin{equation}\label{equ Betti Exc} b^{\star}(X)=b^{\star}(\frac{\mathbb{C}\times K3}{\langle \sigma \rangle})+b^{\star}(Exc)-b^{\star}(\mathcal{C}). \end{equation}
Consequently, 
  \begin{eqnarray}
& &  b^{2}=rkNS(\sigma)+\textrm{Number of all irreducible fixed curves},  \label{equ b2}
\\& & b^{3}=2\Sigma_{\mathcal{C}_{i}\ \textrm{is an irreducible fixed curve}}g_{\mathcal{C}_{i}},\label{equ b3}
\\& & b^{4}=1+\textrm{Number of all irreducible fixed curves}.  \label{equ b4}
\end{eqnarray}
Moreover, $b^{i}=0$ when $i\geq 5$. 
\end{prop}
\begin{proof} By the De Rham cohomology version of the double Mayer-Vietoris argument \cite[4.6 blow up sub-manifolds]{GH},  \cite[I.2]{BT} (also see \cite[4.1]{DL} for the cohomology ring of blow up),    we find the following $\sigma-$invariant isomorphism
\begin{equation}\frac{H^{i}[Bl_{\mathcal{C}}(\mathbb{C}\times K3)]}{H^{i}[\mathbb{C}\times K3]}= \frac{ H^{i}[\widetilde{N}_{\mathcal{C}}]}{H^{i}[N_{\mathcal{C}}]},
\label{equ coh of resol}
\end{equation}
where $N_{\mathcal{C}}$ is a $\sigma-$invariant tubular neighborhood of the fixed curve $\mathcal{C}$,   and $\widetilde{N}_{\mathcal{C}}$ is the blow up of  $N_{\mathcal{C}}$.  The injections  
$$H^{i}[\mathbb{C}\times K3]\longrightarrow H^{i}[Bl_{\mathcal{C}}(\mathbb{C}\times K3)],\ \   H^{i}[N_{\mathcal{C}}]\longrightarrow H^{i}[\widetilde{N}_{\mathcal{C}}]$$
are simply given by  pullback of a closed $i-$forms under the smooth map. Therefore,  they are  $\sigma-$invariant.  Denoting the $\sigma-$invariant sub-spaces of the involved cohomologies by $H_{\sigma}^{i}$,  the $\sigma-$invariant sub-spaces of the quotient spaces  \eqref{equ coh of resol} are also isomorphic:
\begin{equation}\frac{H_{\sigma}^{i}[Bl_{\mathcal{C}}(\mathbb{C}\times K3)]}{H_{\sigma}^{i}[\mathbb{C}\times K3]}= \frac{ H_{\sigma}^{i}[\widetilde{N}_{\mathcal{C}}]}{H_{\sigma}^{i}[N_{\mathcal{C}}]}.
\label{equ coh of resol 0}
\end{equation}
Satake isomorphism \cite{Satake,ALR} implies  $\sigma-$invariant  subspace of the cohomology is isomorphic to the singular cohomology of the quotient orbifold i.e.  
\begin{equation}\frac{H^{i}[X]}{H^{i}[\frac{\mathbb{C}\times K3}{\langle \sigma \rangle}]}= \frac{H^{i}[\frac{\widetilde{N}_{\mathcal{C}}}{\langle \sigma \rangle}]}{H^{i}[\frac{N_{\mathcal{C}}}{\langle \sigma \rangle}]}. 
\label{equ coh of resol 1}
\end{equation}




The neighborhood $\frac{N_{\mathcal{C}}}{\langle \sigma \rangle}$ retracts to $\mathcal{C}$ and  $\frac{\widetilde{N}_{\mathcal{C}}}{\langle \sigma \rangle}$ retracts to $Exc$ because $\sigma$ is lifted to the blow-up.  Therefore 
\begin{equation}\frac{H^{i}[X]}{H^{i}[\frac{\mathbb{C}\times K3}{\langle \sigma \rangle}]}= \frac{H^{i}[Exc]}{H^{i}[\mathcal{C}]}.  
\label{equ coh of resol 2}
\end{equation}

\begin{itemize}\item For $H^{\star}[\frac{\mathbb{C}\times K3}{\langle \sigma \rangle}]$, the obvious radial retraction of $\mathbb{C}-$component says 
$$H^{i}[\frac{\mathbb{C}\times K3}{\langle \sigma \rangle}]=H^{i}[\frac{K3}{\langle \sigma \rangle}]. $$
$H^{2}[\frac{K3}{\langle \sigma \rangle}]$ is simply the rank of $\sigma-$invariant N\'eron-Severi lattice \cite[Theorem 2.1]{AST} i.e. an integer cohomology class 
is $\sigma-$invariant only if it is in N\'eron-Severi lattice. Satake isomorphism again says  $H^{0}[\frac{K3}{\langle \sigma \rangle}]$ and $H^{4}[\frac{K3}{\langle \sigma \rangle}]$  are both $1-$dimensional,  and both $H^{1}[\frac{K3}{\langle \sigma \rangle}],\ H^{3}[\frac{K3}{\langle \sigma \rangle}]$  vanish. 

\item For $H^{\star}(Exc)$, 
Leray-Hirsch theorem \cite{BT} says
\begin{equation}\label{equ LH}H^{\star}(Exc)=H^{\star}(\mathcal{C})\otimes H^{\star}(\mathbb{P}^{1})\end{equation}
as  in K\"unneth-formula for trivial fibration, because the Euler class  of the  universal bundle  generates  $H^{2}$ of each $\mathbb{P}^{1}-$fiber. 
\end{itemize}

With \eqref{equ coh of resol 2}, the above calculation for  $H^{\star}[\frac{K3}{\langle \sigma \rangle}]$ and $H^{\star}(Exc)$ proves  formulas  \eqref{equ b2}--\eqref{equ b4}. The reason why there is factor of ``$2$" in  \eqref{equ b3} is that $b^{1}$ of a curve is twice of the genus. \end{proof}

Let  $L_{K3,j},\ j=1...rkNS(\sigma)$ be $\sigma-$invariant line bundles that form a basis of $NS(\sigma)\otimes \mathbb{R}$ as a vector space,  such that $\frac{L_{K3,j}}{2}$ form a basis for the free part of the $\mathbb{Z}-$module $NS(\sigma)$.  $L_{K3,j}$ descends to $\frac{K3}{\langle \sigma \rangle}$ i.e.  it is the pullback of a line bundle on $\frac{K3}{\langle \sigma \rangle}$.  
Consider  a crepant resolution $\frac{[\widetilde{P^{1}\times K3}]_{0,\infty}}{\langle \sigma \rangle}$ with a K\"ahler metric $\omega_{c}$,  such that both singularities on  fibre at $0$ and $\infty$ are resolved.    This is not the orbifold compactification $\frac{\widetilde{P^{1}\times K3}}{\langle \sigma \rangle}$ throughout that only resolves singularities on fibre at $0$.     

The pullback sheaves  $\pi_{3}^{\star}L_{K3,j}$ are locally free (see \eqref{equ pi3}).   Thus they admit smooth Hermitian metrics and curvature forms $\theta_{L_{K3,j}}$.  
\begin{equation}\label{equ pi3}\pi_{3}:\ \frac{[\widetilde{P^{1}\times K3}]_{0,\infty}}{\langle \sigma \rangle}\longrightarrow \frac{P^{1}\times K3}{\langle \sigma \rangle} \longrightarrow \frac{K3}{\langle \sigma \rangle}. 
\end{equation}
The first arrow from left is the resolution morphism, and the second is projection to the base of the iso-trivial $\mathbb{P}^{1}-$fibration. 
Restriction on a $K3-$fibre away from the fibres at $0$ and $\infty$ and the  prime (irreducible) exceptional divisors  on $X$ (which do not intersect)  show that the Chern classes of 

$$[Exc_{i}],\  i=1,..,l_{0},\ \ L_{K3,j},\ j=1...rkNS(\sigma)$$
are linearly independent in  
$H^{2}[\frac{[\widetilde{P^{1}\times K3}]_{0,\infty}}{\langle \sigma \rangle},\mathbb{R}]$ and $H^{2}[X,\mathbb{R}]$ (as de Rham cohomology),  where $l_{0}$ is the number of irreducible fixed curves on the $K3$,  and $[Exc_{i}]$ are still the exceptional divisors of singularities on fibre at $0$,  not at $\infty$.  
Let $\theta_{[Exc_{i}]}$ be curvature forms of smooth metrics on these line bundles $[Exc_{i}]\longrightarrow \frac{[\widetilde{P^{1}\times K3}]_{0,\infty}}{\langle \sigma \rangle}$.  The K\"ahler cone on $X$ is open at  the restricted cohomology class represented by the arbitrary K\"ahler metric  $\omega_{c}$ on the compactification i.e. 
$$\omega_{c}+\Sigma_{i=1,..,l_{0}} t_{i}\theta_{[Exc_{i}]}+\Sigma_{j=1..rkNS(\sigma)}t_{j}\theta_{L_{K3,j}}$$
is still K\"ahler when the co-efficients $t_{i},\ t_{j}$ have small enough absolute values.  Particularly, by \eqref{equ b2}, number of parameters of  the family of examples provided by Theorem \ref{Thm} on a single $K3$ equals  $b^{2}$, and  $H^{2}[X,\mathbb{R}]=H^{1,1}[X,\mathbb{R}]$. 
 \subsection{Table of  examples and proof of Corollary \ref{Cor 2} \label{Sect examples}}
  By Proposition \ref{prop Betti numbers} and  lattice data of $K3-$surfaces with a non-symplectic Nikulin involution \cite{Nikulin0, Nikulin,Nikulin1}, we calculate the Betti numbers, and count the number of parameters for $ALG_{\infty}(\pi)$ Ricci flat manifolds. 
     
    In Table \ref{Table 1} -- \ref{Table 3} below,  the column ``$rk NS(\sigma)$" is the rank of the $\sigma-$invariant sub-lattice in  N\'eron-Severi lattice. The column ``largest genus $g$" is the largest genus of a  fixed curve.  In table 1 and 2,  the column ``number of rational curves" is the number of irreducible fixed curves counted without the largest genus curve, which must  be rational.  This  is not the case for the exceptional examples from  Table \ref{Table 3}.  
 
 In conjunction with Betti number Proposition  \ref{prop Betti numbers}, the column ``$b^{2}_{local}$" means the local contribution  from the resolution to the $(1,1)-$Hodge number which is the same as the  $b^{2}$ of the resolution. It equals number of  all fixed curves (connected components of the fixed locus). The $b^{2},\ b^{3},\ b^{4}$ mean the Betti numbers of $X=\widetilde{\frac{\mathbb{C}\times K3}{\langle \sigma \rangle}}$.  The  column ``Dim $\mathcal{M}_{\sigma}$" means dimension of moduli of $K3$ surfaces that equals $20-rkNS(\sigma)$. The column ``ALGs" is the number of parameters of the ALG Ricci flat metrics family obtained, which  equals $b^{2}+Dim \mathcal{M}_{\sigma}$+1. The ``1" means the scaling factor $\lambda$ \eqref{equ model orb product metric}. 
     \begin{itemize}
     \item The prime exceptional divisor corresponding to  the  irreducible fixed curve of  genus $g$ (denoted by $\mathcal{C}(g)$ in the first column)  is the  projective bundle 
$\mathbb{P}[O_{\mathcal{C}(g)}\oplus T^{1,0}_{\mathcal{C}(g)}]$.  
\item The prime exceptional divisor corresponding to a (irreducible) rational curves in column 2 is  Hirzebruch surface $\Sigma_{2}$.

\item  For the exceptional examples in  Table \ref{Table 3} below, the exceptional divisor is the disjoint union of the trivial $\mathbb{P}^{1}-$fibrations $E_{1}\times \mathbb{P}^{1}$ and $E_{2}\times \mathbb{P}^{1}$,  where $E_{i}$ are the irreducible fixed  elliptic curves. 
     \item In all 3 tables, the compactifying divisor at $\infty$ is $\frac{\textrm{the K3}}{\langle \sigma \rangle}$, which is always a smooth rational surface. 
     \end{itemize} 
       The calculations are computer aided. The sum of the last columns of all tables below is 1638. All triples of Betti numbers realized by  Table 1 and 2 are distinct. There are 64 rows in total in Table 1 (36 rows) and 2 (28 rows). Table 3 does not give new triple of Betti number as it coincides with one row in Table 1. The counting for Corollary \ref{Thm} is complete.  
       
       Two complete Ricci flat manifolds $(X,\omega_{RFALG})$, $(X^{\prime},\omega^{\prime}_{RFALG})$ in  Corollary \ref{Cor 2} are said to be   weakly distinct if one of the following holds. 
\begin{itemize}\item  The generic $K3$ fibers are  not isomorphic as $\rho-$polarized $K3-$surfaces, in the sense as \cite{AST} (``similar" to that in \cite{DK}).
\item The generic $K3$ fibers are isomorphic as  $\rho-$polarized $K3-$surfaces (which uniquely yields an isomorphism $F:\ X\longrightarrow X^{\prime}$ between the crepant resolutions),  but  the K\"ahler classes are different i.e. $F^{\star}[\omega^{\prime}_{RFALG}]\neq \omega_{RFALG}$. 
\item The generic $K3$ fibers are isomorphic as   $\rho-$polarized $K3-$surfaces, and the K\"ahler classes are ``the same" i.e.  identified by the isomorphism, but the $\lambda$ factors in \eqref{equ model orb product metric} are different. 
\end{itemize}

         \begin{center}
\begin{tabular}{|p{0.9cm}|p{1.4cm}|p{1.5cm}|p{0.7cm}|p{0.5cm}|p{0.5cm}|p{0.5cm}|p{1cm}|p{1cm}|}  
  \hline
g & Number of rational Curves &   $rkNS(\sigma)$ & $b^{2}_{local}$ &  $b^{2}$&$b^{3}$ & $b^{4}$ &  Dim $\mathcal{M}_{\sigma}$ & Dim ALG family \\   \hline
10  &0 &     1 &1   &2 & 20 & 2 & 19 & 22\\   \hline
10  &1 &     2 & 2 &4 &20  &3 & 18 & 23\\   \hline
9  &0 &     2 & 1 &3 & 18 & 2 & 18 & 22\\   \hline
8  &0 &     3 &1  & 4& 16 &2  & 17 & 22\\   \hline
9  &1 &     3 &2  & 5& 18 & 3& 17 & 23\\   \hline
7  &0 &     4 &1  &5 & 14 &2 &  16 & 22\\   \hline
8  &1 &     4 &2  & 6& 16 & 3 & 16 & 23\\   \hline
6  &0 &     5 &1  &6 & 12 & 2 & 15 & 22\\   \hline
7  &1 &     5 &2  &7 & 14 &3 & 15 & 23\\   \hline
5  &0 &     6 &1  & 7&10  & 2  & 14 & 22\\   \hline
6  &1 &     6 & 2 &8 & 12 & 3 & 14 & 23\\   \hline
7  &2 &     6 &3  &9 & 14 &4 & 14 & 24\\   \hline
4  &0 &     7 &1  &8 & 8 &2 & 13  &22 \\   \hline
5  &1 &     7 & 2 &9 & 10 &3  & 13  & 23\\   \hline
6  &2 &     7 & 3  &10 &12  & 4& 13  & 24\\   \hline
3  &0 &     8 &1  &9 & 6 & 2 &  12& 22\\   \hline
4  &1 &     8 & 2  & 10& 4 & 3& 12 & 23\\   \hline
5  &2 &     8 &3  & 11& 10 & 4 & 12 & 24\\   \hline
6  &3 &     8 &4  &12 & 12 &5  & 12 & 25\\   \hline
2  &0 &     9 &1  & 10&4  &2  & 11 & 22\\   \hline
3  &1 &     9 &2  & 11& 6 & 3 &  11 & 23\\   \hline
4  &2 &     9 &3  &12 & 8 & 4 & 11 & 24\\   \hline
5  &3 &     9 &4  & 13& 10 &5 &   11 & 25\\   \hline
6  &4 &     9 & 5 &14 & 12 &6 & 11 & 26\\   \hline
1  &0 &     10 & 1 & 11&2  &2 & 10 & 22\\   \hline
2  &1 &     10 &2  & 12& 4 & 3 & 10 & 23\\   \hline
3  &2 &     10 & 3 & 13& 6 & 4 & 10 & 24\\   \hline
4  &3 &     10 & 4 &14 & 8  & 5 & 10 & 25\\   \hline
5  &4 &     10 &5  & 15& 10 &6  & 10 & 26\\   \hline
6  &5 &     10 & 6 & 16&12  &7  & 10 & 27\\   \hline
0  &0 &     11 &1  & 12& 0 & 2 & 9 & 22\\   \hline
1  &1 &     11 &2  &13 & 2 &3 &  9 & 23\\   \hline
2  &2 &     11 &3  & 14& 4 &  4& 9 & 24\\   \hline
3  &3 &     11 & 4 &15 & 6 & 5 & 9 & 25\\   \hline
4  &4 &     11 &5  & 16& 8 & 6  & 9 & 26\\   \hline
5  &5 &     11 &6  &17 & 10 & 7 & 9 & 27\\   \hline
\end{tabular}
\captionof{table}{Non-exceptional examples Part I . Number of parameters obtained is $848$.  36 Distinct triples of Betti numbers $(b^{2},b^{3},b^{4})$ are realized.}
\label{Table 1}
\end{center}

\begin{center}
\begin{tabular}{|p{0.9cm}|p{1.4cm}|p{1.5cm}|p{0.7cm}|p{0.5cm}|p{0.5cm}|p{0.5cm}|p{0.7cm}|p{0.9cm}|}
  \hline
g &Number of rational Curves& $rkNS(\sigma)$ & $b^{2}_{local}$  & $b^{2}$&$b^{3}$ & $b^{4}$ &  Dim $\mathcal{M}_{\sigma}$ & Dim ALG family   \\ \hline
0  &1 &     12 & 2   &14 & 0 & 3  & 8 &23\\   \hline
1  &2 &     12 &3   &15 & 2 & 4&  8 & 24\\   \hline
2  &3 &     12 &4  & 16& 4 & 5 &  8 & 25\\   \hline
3  &4 &     12 & 5  & 17& 6 & 6 & 8 & 26\\   \hline
4  &5 &     12 &6  &18 & 8 & 7&  8 &27\\   \hline
0  &2 &     13 & 3 & 16& 0 & 4 & 7 & 24\\   \hline
1  &3 &     13 &4  & 17& 2 & 5&   7 & 25\\   \hline
2  &4 &    13 &5  &18 &4  & 6 & 7 & 26\\   \hline
3  &5  &     13 & 6 &19 & 6 & 7& 7 & 27\\   \hline
0  &3 &     14 &4  &18 & 0 & 5 & 6 & 25\\   \hline
1  &4 &     14 & 5 &19 &2  & 6 & 6 & 26\\   \hline
2  &5 &     14 &6  &20 & 4 &7 &  6 & 27\\   \hline
3  &6 &     14 & 7  & 21& 6 & 8&  6 & 28\\   \hline
0  &4 &     15 & 5& 20 & 0 & 6 & 5 & 26\\   \hline
1  &5 &     15 &6  &21 & 2 & 7 &  5 & 27\\   \hline
2  &6 &     15 & 7 &22 & 4 &8  &  5 & 28\\   \hline
0  &5 &     16&6   & 22& 0 & 7 & 4 & 27\\   \hline
1  &6 &     16&7  &23 & 2 &8 &  4 & 28\\   \hline
2  &7 &     16&8  &24 & 4 & 9 &  4 & 29\\   \hline
0  &6 &     17&7  & 24&0  & 8 &  3& 28\\   \hline
1  &7 &     17&8  & 25& 2 & 9 &  3 & 29\\   \hline
2  &8 &     17& 9 &26 &  4&10   &3 & 30\\   \hline
0  &7 &     18 &8   & 26&0  & 9 & 2&29 \\   \hline
1  &8 &     18 & 9 & 27& 2 &10  & 2 & 30\\   \hline
2  &9 &     18 &10  & 28&4  & 11 &  2& 31\\   \hline
0  &8 &     19&9   &28 &  0& 10 & 1 & 30\\   \hline
1  &9 &     19& 10 &29 & 2 &11 & 1 & 31\\   \hline
0  &9 &     20&10  & 30& 0 & 11 & 0 &31\\   \hline
\end{tabular}\captionof{table}{Non-exceptional examples Part II.   Number of parameters obtained is 767.  28  distinct triples of Betti numbers $(b^{2},b^{3},b^{4})$ is realized (beyond Table 1).}\label{Table 2}
\end{center}

   \begin{center} 
    \begin{tabular}{|p{1.7cm}|p{0.7cm}|p{0.5cm}|p{0.5cm}|p{0.5cm}|p{0.7cm}|p{1.1cm}|} 
  \hline
$rkNS(\sigma)$ & $b^{2}_{local}$  & $b^{2}$&$b^{3}$ & $b^{4}$ & Dim $\mathcal{M}_{\sigma}$& Dim ALG family \\   \hline
  10 & 2   & 12 &4 & 3 & 10 & 23    \\   \hline
\end{tabular}\captionof{table}{Exceptional examples. Fixed locus consists of two disjoint elliptic curves. Triple of Betti numbers $(b^{2},b^{3},b^{4})=(12,4,3)$ coincides with a row in Table 1.} \label{Table 3}\end{center} 

\section{Appendix}
\subsection{Poincare inequality under model radius exhaustion}
Instead of   distance function of the whole manifold,  the argument in \cite[Prop 4.8 i,\ ia,\ ib]{HeinThesis} for the initial decay works verbatim for the model distance $r$ on the end. The reason why we  need this is that our end is a twisted product. It seems harder for us to show the annulus are connected, than to work with the model distance $r$. For the Poincare inequality and initial decay,  we essentially only need the reference metric to be equivalent to the model orbifold metric on the end. 
\begin{lem}\label{lem Poincare} Let $A=\{\rho\leq  r \leq \rho^{\prime}\}$, $\rho^{\prime}>\rho>0$ be a non-trivial annulus in the  cone with (smooth) closed link $\mathbb{S}$: $$\mathbb{S}\times \mathbb{R}^{+},\ \ \ dr^{2}+r^{2}(d\mathbb{S})^{2}.$$
Let $Y$ be a (smooth) closed manifold, and  $P_{A},\ P_{Y}$ be the Neumann Poincare constant on the annulus and  $Y$ respectively. The following holds for any function $u\in C^{2}[A\times Y]$. $$\int_{A}\int_{Y}u^{2}\leq (P_{A}+P_{Y})\int_{A}\int_{Y}|\nabla u|^{2}.$$
\end{lem}
\begin{proof} Write $u_{Ave,Y}$ for $\frac{\int_{Y}u}{Vol(Y)}$ and $v$ for $u- u_{Ave,Y}$. Then  $u_{Ave,Y}$ is a function in $A$ only, and is constant along each fiber.  $v(a,\cdot )$ has $0$ integral/average on any fiber $Y$, for any $a  \in A$. Obviously, $u_{Ave,Y}$ and $v$ are perpendicular under $L^{2}[Y]-$inner product, and
$$u=u_{Ave,Y}+v.$$
We calculate
\begin{equation}\label{equ almost Poincare}
\int_{A}\int_{Y}u^{2}=\int_{A}\int_{Y}u^{2}_{Ave,Y}+\int_{A}\int_{Y}v^{2}\leq P_{A}\int_{Y}\int_{A}|\nabla_{A}u_{Ave,Y}|^{2}+P_{Y}\int_{A}\int_{Y}|\nabla_{Y}v|^{2}.
\end{equation}
We  have the following orthogonality because $\nabla_{A}$ obviously commutes with integration in $Y$, and $u_{Ave,Y}$ is a constant along each fiber $\{a\}\times Y$. \begin{equation}\label{equ L2 splitting}
\int_{Y}\nabla_{A}u_{Ave,Y}\cdot \nabla_{A}v=0.
\end{equation}

Hence \eqref{equ L2 splitting} implies 
\begin{equation}
\int_{A}\int_{Y}|\nabla_{A}u|^{2}=\int_{A}\int_{Y}|\nabla_{A}u_{Ave,Y}|^{2}+\int_{A}\int_{Y}| \nabla_{A}v|^{2}.
\end{equation}
Plugging the above into \eqref{equ almost Poincare} and throwing away the ``good" term $$P_{A}\int_{A}\int_{Y}|\nabla_{A}(u- u_{Ave,Y})|^{2},$$ the proof is complete. 

For the reader's convenience, we provide the routine proof for \eqref{equ L2 splitting}. Let $a\in A$ and $a^{i},\ i=1...dimA$ be  a normal coordinate at/near $a$ under the cone metric. 
Then $\frac{\partial u_{Ave,Y}}{\partial a^{i}}$ is still a constant along the fiber $\{a\}\times Y$. On the other hand, because 
$$\int_{Y}v\equiv 0\ \textrm{for any}\ b\in A,$$
we find $$0=\frac{\partial}{\partial a^{i}}\int_{Y}v=\int_{Y}\frac{\partial v}{\partial a^{i}}.$$
Finally, we multiply them to get 
$$ \int_{Y} \frac{\partial u_{Ave,Y}}{\partial a^{i}}\cdot\frac{\partial v}{\partial a^{i}}=\frac{\partial u_{Ave,Y}}{\partial a^{i}}\cdot \int_{Y}\frac{\partial v}{\partial a^{i}}=0.$$
This is similar to the integration along fiber in \cite{BT}. 
\end{proof}

\subsection{Hein energy decay estimate for model radius function}
Define $E(\rho)=\int_{\{r\geq \rho\}}|\nabla \phi|^{2}$. The  argument in \cite[Prop 4.8 i,\ ia,\ ib]{HeinThesis} says
there exists $0\leq \mu<1$ such that 
\begin{equation}\label{equ energy decay}
E(2^{n}r_{0})\leq \mu^{n}E(r_{0}). 
\end{equation}

Standard iterative argument (see \cite[Proof of Lemma 8.23]{GT} for a sophisticated version)  implies polynomial energy decay i.e. 
\begin{equation}\label{equ Dirichlet energy decay}
E(r)\leq C_{\beta_{0}}r^{-\beta_{0}},
\end{equation}
where $C_{\beta_{0}}=E(r_{0})\cdot r_{0}^{-\frac{\log \mu}{\log 2}},\ 
\beta_{0}=-\frac{\log \mu}{\log 2}$.

\begin{proof}[Hein's proof of \eqref{equ energy decay} under the model radius function:].\\

 The energy inequality \cite[(4.14)]{HeinThesis} gives 
\begin{equation}
\int_{\{r\geq 2^{i+1}r_{0}\}}|\nabla \phi|^{2}\leq \frac{C}{2^{i}r_{0}}|\phi-\phi_{Ave}|_{L^{2}[2^{i}r_{0}\leq r\leq 2^{i+1}r_{0}]}|\nabla\phi|_{L^{2}[2^{i}r_{0}\leq r\leq 2^{i+1}r_{0}]}.
\end{equation}
The radius $r_{0}$ is chosen large enough such that $f=0$ on $\{r\geq \frac{r_{0}}{2}\}$. Therefore the $f-$term in  \cite[(4.14)]{HeinThesis} vanishes. The annulus $2^{i}r_{0}\leq r\leq 2^{i+1}r_{0}$ are scaling equivalent  to the standard $\{1\leq r \leq 2\}$.  By Lemma \eqref{lem Poincare} and the Poincare constant bound $$P_{A}+P_{Y}\leq C(2^{i}r_{0})^{2},$$   we find 
\begin{equation*}
\int_{\{r\geq 2^{i+1}r_{0}\}}|\nabla \phi|^{2}\leq C\int_{\{2^{i}r_{0}\leq r\leq 2^{i+1}r_{0}\}}|\nabla\phi|^{2}=C(\int_{\{r\geq 2^{i}r_{0}\}}|\nabla\phi|^{2}-\int_{\{ r\geq 2^{i+1}r_{0}\}}|\nabla\phi|^{2}).
\end{equation*}
This means 
\begin{equation*}
\int_{\{r\geq 2^{i+1}r_{0}\}}|\nabla \phi|^{2}\leq \frac{C}{C+1}\int_{\{r_{0}\geq  2^{i}r_{0}\}}|\nabla\phi|^{2}.
\end{equation*}
This means
\begin{equation}\label{equ energy decay step 1} E(2^{i+1}r_{0})\leq \mu\cdot E(2^{i}r_{0})\ \textrm{for any}\ i\geq 0,\end{equation}
 if we take $\mu=\frac{C}{C+1}<1$. Then \eqref{equ energy decay} follows from \eqref{equ energy decay step 1}. 
\end{proof}
\subsection{Tian-Yau quasi atlas and Schauder norms\label{Appendix Tian Yau atlas}}

\subsubsection{Review}
Unlike injectivity radius and harmonic radius,  the concept of ``holomorphic radius " does not seem  immediate and obvious. We apply the quasi atlas \cite{TY1, HeinThesis} to the compact K\"ahler manifold $Y$, even though it is mainly applied to non-compact K\"ahler manifolds.
By Hein's version of Tian-Yau  quasi atlas \cite[Lemma 4.3]{HeinThesis}, there is a $C^{1,\alpha}$ quasi atlas which is also a $C^{k,\alpha}-$quasi atlas, for any $k>0$ and $\alpha\in (0,1)$. 
Henceforth we refer to it as ``the quasi atlas".  In  Tian-Yau coordinate, the crucial properties we rely on here is the quasi-isometry between the pullback $\Phi_{\underline{x}}^{\star} \omega_{0}$  and the Euclidean metric, and the uniform $C^{k,\alpha}$ bound on $\Phi_{\underline{x}}^{\star} \omega_{0}$.

Namely,   there is a $\mu_{0}>0$  such that for any $x\in Y$, there is a $\underline{x}\in A\subset Y$ such that 
$x\in \Phi_{\underline{x}}(B_{euc})$ and 
 \begin{itemize}
 \item $B_{euc,\Phi^{-1}_{\underline{x}}(x)}(\mu_{0})\subset B_{euc,O}(1)$,
 \item $dist[B_{euc,x}(\mu_{0}), \partial B_{euc,O}(1)]> \mu_{0}$.
 \end{itemize}
$A$ is the subset of ``centers" of open sets in the quasi-atlas,  $\Phi_{\underline{x}}$ is the holomorphic coordinate map, $B_{euc,x}(r)$ is the unit ball in Euclidean space 
$\mathbb{C}^{d-1}$ with radius $r$ centered at $x$, and $\Phi_{\underline{x}}$ identifies the origin to $\underline{x}$ on the manifold.  We interpret the conditions in \cite{HeinThesis} to coordinate chart. A schematic description of $x$, $\underline{x}$, and a point $y$ with $dist(x,y)<\mu_{0}$ under the coordinate in quasi-atlas is as Figure \ref{fig balls}. 
\begin{figure}[h]
\begin{center}
\begin{tikzpicture}[scale=1.5]
  \draw (0,0) circle (1.5);
   \node  at (0,0) {$\underline{x}$};
    \node  at (55:0.8) {$x$};
     \draw (55:0.8) circle (0.5);
      \node  at (45:1.1) {$y$};
   \end{tikzpicture} \end{center}\caption{Positions of $x,\ y$ and $\underline{x}$}\label{fig balls}
 \end{figure}
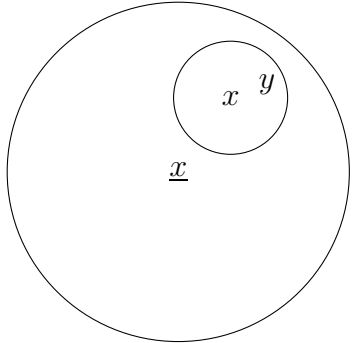
\subsubsection{Poly-Cylinder}
Let $$\mu_{1}<< \frac{\min\{\mu_{0},h_{0,Y},inj_{0,Y}, conv_{0,Y}\}}{c_{TY,ALG}}$$ be a positive  radius where $c_{TY,ALG}$ is a large enough constant depending on   ALG model  \eqref{equ model orb product metric} and Tian-Yau atlas, where $h_{0,Y},\ inj_{0,Y},\ conv_{0,Y}$ are  harmonic,  injectivity, and convex  radius of $Y$ under the metric $\omega_{Y}$ respectively. Let the Schauder norm on the fibred annulus $A(R_{1},R_{2})$ be defined in the standard way (for example, see \cite[1.2.2]{Joyce}): supremum of norms of co-variant derivatives under  ALG model, together with $\alpha-$H\"older semi norms by the distance  on $\mathbb{C}\times Y$. To estimate  the Lipschitz quotient $\frac{|v(x)-v(y)|}{d^{\alpha}(x,y)}$, where  $v$ is either a function or a  matrix valued function representing  a $(1,1)-$form under coordinate chart, it is well known that it suffices to assume $d(x,y)<\mu$ for some uniform number $\mu$ independent of $v$,  because otherwise the Lipschitz quotient is bounded by the $C^{0}-$norm of $v$. Therefore, \textit{it suffices to estimate in poly-cylinders} 
\begin{equation}\label{equ TY polycylinder}PC_{z, \underline{x},y}(\mu_{1})\triangleq D_{z, \mathbb{C}}(\mu_{1})\times [B_{euc, y}(\mu_{1})]\subseteq \mathbb{C}\times Y\end{equation}  in Tian-Yau chart. The symbol $D_{z, \mathbb{C}}(\mu_{1})$ denotes disk/ball of radius $\mu_{1}$ centered at $z$ in the $\mathbb{C}-$component. 
When $\mu_{2}< < \mu_{1}$, the above poly-cylinder contains 
the Euclidean ball $B_{euc, (z,y)}(\mu_{2})$ where the sub-scripts are usually suppressed. In the energy decay estimate \eqref{equ energy decay 0}, Neumann-Poincar\'e inequality is applied on such a ball. 
\subsubsection{Weighted Schauder norm \label{section weighted Schauder norm}}
We define  weighted Schauder norms on the end by supremum of naively weighted norms on fibred annulus:  
\begin{equation}|v|_{C^{k,\frac{1}{2}}_{\beta}[A(R,\infty)]}\triangleq \sup_{l\geq R} l^{\beta}|v|_{C^{k,\frac{1}{2}}[A(l,l+2)]}.
\end{equation}
On one hand, 
$$|v|_{C^{k,\frac{1}{2}}[A(l,l+2)]}\leq C\sup_{{PC_{z, \underline{x},y}(\mu_{1})}\cap [A(l,l+2)] \neq \emptyset}  |v|_{C^{k,\frac{1}{2}}[PC_{z, \underline{x},y}(\mu_{1})]}.$$
The width $2$ of the annulus is not serious. Any universal positive constant will do. 
On the other hand, when $\Phi_{\underline{x}}[PC_{z, \underline{x},y}(\mu_{1})]\subset A(l,l+2)$ (the map in $\mathbb{C}-$component is identity), we have 
$$|v|_{C^{k,\frac{1}{2}}[A(l,l+2)]}\geq   C|v|_{C^{k,\frac{1}{2}}[PC_{z, \underline{x},y}(\mu_{1})]}.$$
The norm on the right  means the (usual) norm under Tian-Yau coordinate (and Euclidean metric therein), but the norm on the left employs geodesic distance as in \cite{Joyce}.

\begin{proof}[Proof of \eqref{equ higher order initial decay}:]
The weighted $L^{2}-$estimate for  Newtonian potential \cite[Theorem 1.2. (2)]{West},  and exponential decay of $O(e^{\tau_{0}}r)-$harmonic functions with fiber-wise $0-$average \cite[Theorem 1.2. (5)]{West},  in conjunction with the decomposition \eqref{equ decomposition}
\eqref{equ decomposed},  yields the $L^{2}-$decay for the part of the potential with fiber-wise $0-$average: 
\begin{equation}|\phi_{0}|_{L^{2}[(a,b)\times (\theta_{0},\theta_{1})\times Y]}\leq \frac{C}{r^{\delta_{0}}}. 
\end{equation}
In view of \eqref{equ ddbar of phi est equiv to both parts} and the discussion below it saying decay of $i\partial\overline{\partial}\phi$ implies decay of $i\partial\overline{\partial}\phi_{0}$,   Moser iteration gives 
$$|\phi_{0}|_{C^{0}[D(\mu_{3})]}\leq \frac{C}{r^{\delta_{0}}},$$ 
where $\mu_{3}\leq \frac{1}{10}\min\{\mu_{1},\ \mu_{2}\}$ is also required to be small enough  with respect to  $R$ such that $D(1000^{d}\mu_{3})$ is contained in the fibered sector $(a,b)\times (\theta_{0},\theta_{1})\times Y$ in  \eqref{equ 1 Schwarts strip}.

For the part $\underline{\phi}$, since it is pullback from (the domain in) the $\mathbb{C}-$component,  we can simply solve  
\begin{equation}\Delta_{\mathbb{C}}g=\Delta_{\mathbb{C}}\underline{\phi}\ \textrm{on}\  D(\frac{\mu_{3}}{2}),\ g|_{\partial D(\frac{\mu_{3}}{2})}=0\ \textrm{for}\ g.\ \textrm{Thus}\ \
\label{equ decay estimate of g}|g|_{C^{0}[D(\frac{\mu_{3}}{2})]}\leq \frac{C}{r^{\delta_{0}}}.\end{equation}

 The point is still that  a harmonic function of $1-$complex variable is pluri-harmonic. We find 
$$i\partial \overline{\partial}\phi=i\partial \overline{\partial}(g+\phi_{0}). $$
Evans-Krylov estimate \cite{Evans,Krylov} implies
\begin{equation}  |i\partial \overline{\partial}(g+\phi_{0})|_{C^{\alpha_{0}}[D(\frac{\mu_{3}}{3})]}=|i\partial \overline{\partial}\phi|_{C^{\alpha_{0}}[D(\frac{\mu_{3}}{3})]}\leq C\ \textrm{for some}\ \alpha_{0}\in (0,1). 
\end{equation}
Here we at most need $|g|_{C^{0}[D(\frac{2\mu_{3}}{5})]}\leq C$ which is not the full strength of \eqref{equ decay estimate of g}. Therefore,  differentiating the Monge-Ampere equation 
\begin{equation}
[\omega_{0}+i\partial \overline{\partial}(g+\phi_{0})]^{d}=(\omega_{0}+i\partial \overline{\partial}\phi)^{d}=\omega^{d}_{0}\ \ \ \  (\eqref{equ Target MA}\ \ \textrm{on the end})
\end{equation}
under Tian-Yau poly-cylinder chart \eqref{equ TY polycylinder} (with smaller component-wise radius)
 as usual (where the  derivatives of the coordinate forms vanish) yields  higher-order estimate for any $k\geq 0$
\begin{equation}\label{equ ddbar g+phi0}  |i\partial \overline{\partial}(g+\phi_{0})|_{C^{k, \alpha_{0}}[D(\frac{\mu_{3}}{4})]}=|i\partial \overline{\partial}\phi|_{C^{k, \alpha_{0}}[D(\frac{\mu_{3}}{4})]}\leq C_{k}. 
\end{equation}
Here,  $i\partial \overline{\partial}\phi_{\omega_{0}}=\omega_{0}$ is solvable on the  poly-cylinder $D_{z, \mathbb{C}}(\mu_{3})\times [B_{euc, y}(\mu_{3})]$ by  a smooth solution
$\phi_{\omega_{0}}$  on which the $C_{k}'$s above depend.  The estimates \eqref{equ ddbar g+phi0} guarantee that the $C^{k, \alpha_{0}}[D(\frac{\mu_{3}}{4})]-$norm of the co-efficient of 
\begin{equation}\label{equ MA as Laplace}\Delta_{\omega_{g+\phi_{0}}}(g+\phi_{0})=0\ i.e.   \end{equation} the Monge-Ampere equation \eqref{equ Kovalev form of MA} written as the homogeneous ``Laplace" equation,  are bounded.  Linear interior Schauder estimate \cite{GT} for homogeneous equations (that only requires the $C^{0}-$bound on the solution) yields 
$$|g+\phi_{0}|_{C^{k,\alpha_{0}}[D(\frac{\mu_{3}}{10})]}\leq |g+\phi_{0}|_{C^{0}[D(\frac{\mu_{3}}{2})]}\leq \frac{C}{r^{\delta_{0}}}. $$
  The equations  \eqref{equ MA as Laplace} and  \eqref{equ Kovalev form of MA} are the same (on the end) written in different forms. 
  \end{proof}
\subsection{Gluing for ansatz}
\begin{lem}\label{lem Gluing for ansatz}(Gluing for ansatz \cite{HeinJAMES, HeinThesis, HHN}) Under Theorem \ref{Thm}, on  an  iso-trivial ALG Calabi-Yau fibration $\widetilde{\frac{\mathbb{C} \times Y}{\langle  \sigma \rangle}}$, let $\chi(r)$ be a cutoff function $\equiv 1$ when $r\geq R+1$ but $\equiv 0$ when $r\leq R$, for some $R>>1$. Let $\varphi$ be another cutoff function that $\equiv 1$
 when $r\geq R-1$ but $\equiv 0$ when $r\leq R-2$.  Let $b_{0}\triangleq \varphi(r)\frac{i}{2}dzd\overline{z}$. 
 Suppose $\widetilde{\omega}$ is a real $(1,1)-$form on $X$  that is positive/K\"ahler on $\{r\leq R+1\}$, and $\phi$ is a real-valued function on $\{r> R-3\}$ such that
 $ \widetilde{\omega}+i\partial \overline{\partial}\phi$ is positive/K\"ahler on $\{r \geq R\}$. Then  when $t$ is large enough, the real closed $(1,1)-$form
 \begin{equation}\label{equ isotrivial ansatz ALG}
 \underline{\omega}\triangleq \widetilde{\omega}+i\partial \overline{\partial}(\chi \phi)+tb_{0}
 \end{equation}
 is a K\"ahler metric  on $X$ i.e. it is positive definite every-where. 
\end{lem}
$\textrm{Domain of}\ \phi\supset  \overline{Supp \varphi} \supseteq \{\varphi\equiv 1\}\supset  \overline{Supp\chi}$. See  Figure \ref{picture A1}. 

\begin{figure}[h]\begin{center}  \begin{tikzpicture}[line/.style={-}]
\draw[->,semithick] (-6,0) -- (6,0);
\node at (-5.5,0) [circle,fill,inner sep=1pt]{};
\node at (-5.5,-0.3) {\tiny $R-3$};
\node at (-2.8,0) [circle,fill,inner sep=1pt]{};
\node at (-2.8,-0.3) {\tiny $R-2$};
\node at (-0.1,0) [circle,fill,inner sep=1pt]{};
\node at (-0.1,-0.3) {\tiny $R-1$};
\node at (2.6,0) [circle,fill,inner sep=1pt]{};
\node at (2.6,-0.3) {\tiny $R$};
\node at (5.3,0) [circle,fill,inner sep=1pt]{};
\node at (5.3,-0.3) {\tiny $R+1$};
\path [line, bend right] (2.6,0) edge  (3.95, 1.35);
\path [line, bend left] (3.95, 1.35) edge  (5.3,2.7);
    \draw [dashed,-] (5.3,2.7) -- (6,2.7); 
      \node at (3.7, 1.35) { $\chi$};
\path [line, bend right] (-2.8,0) edge  (-1.45, 1.35);
\path [line, bend left] (-1.45, 1.35) edge  (-0.1,2.7);
    \draw [dashed,-] (-0.1,2.7) -- (2.6,2.7);
    \node at (-1.7, 1.35) { $\varphi$};
     \draw [|-] (-5.5,1) -- (-5, 1); 
       \node at (-4.3, 1) { $Dom(\phi)$};
         \draw [dashed,-] (-3.6,1) -- (-2.5,1);
         \end{tikzpicture}  
\end{center}
\caption{Supports and domain in gluing lemma for ansatz}
\label{picture A1}
\end{figure}
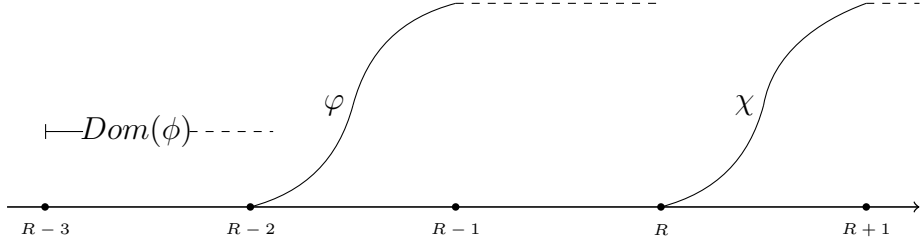
\begin{proof}
It is routine by  positivity of $\widetilde{\omega}$ in the fiber direction and positivity of the bump form $tb_{0}$ in the ``horizontal" direction.
The idea is that in $Supp\chi$/the gluing region, since $\chi$ is pulled back from $\Delta^{\star}$, the $(1,1)-$form $\widetilde{\omega}+i\partial \overline{\partial}(\chi \phi)$ is positive on the fiber $Y-$direction. The mixed forms with both base and fiber directions can be bounded from below by Cauchy-Schwarz inequality. Therefore adding the (horizontal) bump form $tb_{0}$  makes it positive. Here the $t$ is required to be large enough with respect to the smooth $\phi$ which is only required to exist locally. It has bounded $C^{k}-$norms on the gluing region (which has compact closure), for any $k$. 

 We calculate
  \begin{equation}
  \widetilde{\omega}+i\partial \overline{\partial}(\chi \phi)= \widetilde{\omega}+\chi  i\partial \overline{\partial}\phi+\phi  i\partial \overline{\partial}\chi+2Re (i\partial\chi\wedge \overline{\partial} \phi). 
  \end{equation}
  The sum of the first two term on the right side above is positive:
  \begin{equation}
   \widetilde{\omega}+\chi  i\partial \overline{\partial}\phi=(1-\chi)\widetilde{\omega}+\chi (\widetilde{\omega}+ i\partial \overline{\partial}\phi)>0\ \textrm{is positive definite Hermitian}.
  \end{equation}
 We therefore use the norm $|\cdot|$ of the Hermitian metric $H_{\chi}$ (associated to) $\widetilde{\omega}+\chi  i\partial \overline{\partial}\phi$.   Let $v$ and $h$ be a vertical (tangent to fiber) and horizontal $(1,0)-$vector (proportional to $\frac{\partial}{\partial z}$), we compute
  \begin{eqnarray*}& & [\widetilde{\omega}+i\partial \overline{\partial}(\chi \phi)][v+h,\overline{v}+\overline{h}]
  \\&=&  [\widetilde{\omega}+\chi  i\partial \overline{\partial}\phi][v+h,\overline{v}+\overline{h}]+(\phi  i\partial \overline{\partial}\chi)[h,\overline{h}]+ 2Re  (i \partial \chi \wedge \overline{\partial} \phi)[h,\overline{h}]
  \\& &+2Re(i \partial \chi \wedge \overline{\partial} \phi) [h,\overline{v}]+2Re(i \partial \chi \wedge \overline{\partial} \phi) [v,\overline{h}]
  \\&\geq & |v+h|^{2}-c_{1}|v||h|-c_{2}|h|^{2}.
   \end{eqnarray*}
   In  the following bound, the reason for constants $c_{1}$ and $c_{2}$ 
  \begin{eqnarray*}
  & & (\phi  i\partial \overline{\partial}\chi)[h,\overline{h}]+ 2Re  (i\partial \chi \wedge \overline{\partial} \phi) [h,\overline{h}]
  \\& &+2Re (i\partial \chi \wedge \overline{\partial}\phi) [h,\overline{v}]+2Re (i\partial \chi \wedge \overline{\partial}\phi) [v,\overline{h}]
  \\&\geq & -c_{1}|v||h|-c_{2}|h|^{2}.
  \end{eqnarray*}
  is compactness of  $Supp \nabla\chi\cup Supp \partial \chi \cup Supp \overline\partial \chi$. 
   The point is the vanishing of the terms $(\partial \chi \wedge \overline{\partial} \phi) [v,\overline{v}]$, $(\phi  i\partial \overline{\partial}\chi)[v,\cdot]$ and the conjugates i.e. $\partial \chi$ and $\overline{\partial}\chi$ vanishes on vectors tangent to  fiber. Cauchy-Schwarz inequality implies 
   $$-c_{1}|v||h|\geq -\frac{1}{2}|v|^{2}-c_{3}|h|^{2}.$$
   Then when $t>1000(c_{2}+c_{3}+1)$, the bump form $tb_{0}$ makes $$\underline{\omega}\geq \frac{1}{3}H_{\chi}.$$
   \end{proof}
\subsection{``Un-conditional" solvability of Laplace/$i\partial \overline{\partial}$ equation in dimension $1$. }
As a pre-requisite of $i\partial \overline{\partial}$ Lemma \eqref{lem strong iddbar}  and radial K\"ahler potential for radial bump $(1,1)$ form on $\mathbb{C}$ (required in proof of Theorem \ref{Thm}), we need the following.  
 \begin{clm}\label{clm DeltaC solvability Hormander} For any smooth function $f$ on $\Delta^{\star}$, there is a smooth solution $h$ to 
\begin{equation}\label{equ iddbar hormander C}\frac{\partial^{2}h}{\partial u\partial \overline{u}}=f,
\end{equation}
which is equivalent to $\Delta_{\mathbb{C}}h=4f$
 because $\Delta_{\mathbb{C}}=4\frac{\partial^{2}}{\partial u\partial \overline{u}}$.
 
 Suppose $f$ is compactly supported in $\{\rho<\rho_{1}\}$ and radial i.e.  $f=f(\rho)$ is a function of $\rho$ only. Then there is a radial solution $h(\rho)$ such that $h(\rho)=0$ when $\rho> \rho_{1}$.
 \end{clm}

The first part (no symmetry) is a direct consequence of \cite[Theorem 1.4.4]{Hormander} i.e.  the $1-$dimensional $\overline{\partial}-$solvability that implies for any $f\in C^{\infty}(\Delta^{\star})$, 
 $$\frac{\partial \varsigma}{\partial \overline{u}}=f\ \textrm{admits a solution}\ \varsigma \in C^{\infty}(\Delta^{\star}).$$
 The complex conjugate of the same  solvability yields a smooth function $h$ that solves
 $$\frac{\partial h}{\partial u}=\varsigma.$$
We do not need any ($L^{2}$) integrabilty or asymptotes of $f$ near $0$. 

When $f$ is radial, we simply solve the ODE
$$\frac{\partial^{2}h}{\partial u\partial \overline{u}}(\rho)=\frac{1}{4}\Delta_{\mathbb{C}}h(\rho)\ \ i.e.\ \ h^{\prime\prime}+\frac{h^{\prime}}{\rho}=4f$$
by\begin{equation}\nonumber
h=4\log \rho\int^{\rho_{1}}_{\rho}xf(x)dx-4\int^{\rho_{1}}_{\rho}(x \log x)f(x)dx.
\end{equation}
Consequently,  
$h=0$ when $\rho>\rho_{1}$ because the integrals are outside the support of $f$.

\end{document}